\documentclass[11pt]{amsart}
\usepackage{xcolor}

\usepackage[margin=3cm]{geometry}

\usepackage[normalem]{ulem}
\usepackage{amsthm, bbm}
\usepackage{mathtools}
\usepackage{amsmath, amssymb}
\usepackage{mathrsfs}
\usepackage{hyperref}
\usepackage{graphicx, psfrag, epstopdf}
\usepackage[capitalise,noabbrev]{cleveref}
\usepackage{todonotes}
\usepackage{enumerate}
\usepackage{enumitem}

\theoremstyle{definition}
\newtheorem{definition}{Definition}[section]
\newtheorem{theorem}[definition]{Theorem}
\newtheorem{lemma}[definition]{Lemma}

\newtheorem{proposition}[definition]{Proposition}

\newtheorem{remark}[definition]{Remark}

\numberwithin{equation}{section}

\definecolor{OliveGreen}{rgb}{0,0.6,0}

\def\R{\mathbb{R}}

\def\T{\mathbb{T}}

\def\N{\mathbb{N}}
\def\P{\mathbb{P}}

\DeclareMathOperator{\diam}{diam}

\DeclareMathOperator{\Var}{Var}

\title[Ergodic sums over rotations for non-integrable observables]{The influence of the maximal summand on ergodic sums of non-integrable observables over rotations}

\author{Adam Kanigowski}
\address{Department of Mathematics, University of Maryland, College Park, MD USA\newline \indent and\newline 
\indent Faculty of Mathematics and Computer Science, Jagiellonian University \newline
\indent ul. Łojasiewicza 6, 30-348 Krakow, Poland}\email{akanigow@umd.edu}

\author{Tanja I. Schindler}
\address{Faculty of Mathematics and Computer Science, Jagiellonian University \newline
\indent ul. Łojasiewicza 6, 30-348 Krakow, Poland \newline 
\indent and \newline 
\indent Department of Mathematics and Statistics,
University of Exeter \newline
\indent Exeter EX4 4QF, United Kingdom
}
\email{tanja.schindler@uj.edu.pl}
\email{t.schindler@exeter.ac.uk}

\begin{document}

\begin{abstract}
 For $R_{\alpha}$ being an irrational rotation of angle $\alpha$ on the one torus $\T$ and $\phi(x)=\frac{1}{x}-\frac{1}{1-x}$, we compare the behavior of the Birkhoff sum $S_N(\phi)=\sum_{k=0}^{N-1}(\phi\circ R_{\alpha}^k)(x)$ with the successive entry $(\phi\circ R_{\alpha}^N)(x)$. In particular, we are interested in the almost sure limsup behavior of $\frac{(\phi\circ R_{\alpha}^N)(x)}{S_N(\phi)(x)}$. We show that depending on the Diophantine properties of $\alpha$ we have that the limsup either equals $0$ or $\infty$. Moreover, we show that those $\alpha$ for which the limsup equals $0$ form an atypical set in the sense that its Hausdorff dimension equals $\frac{1}{2}$. 
These results have consequences in studying a reparametrization $(T_t)$ of the  linear flow $(L_t)$ with direction $(1,\alpha)$ on the two torus $\T^2$ with function $\varphi$, where $\varphi$ is a smooth non-negative function that has exactly two (non-degenerate) zeros at $\bf p$ and $\bf q$. 
We prove that for a full measure set $(\alpha, {\bf p}, {\bf q})\in \T\times \T^2\times \T^2$ 
the special flow $(T_t)$ exhibits extreme historic behavior proving a conjecture given by Andersson and Guih\'eneuf in \cite{An-Gu}. 
\end{abstract}


\maketitle

\section{Introduction and statement of results}

Let $(\Omega,\mathcal{B},\mu,T)$ be an ergodic,  probability measure preserving dynamical system and let $\varphi:\Omega \to\mathbb{R}_{\geq 0}$. 
 For $x\in \Omega$ denote $S_N(\varphi)(x):=\sum_{k=0}^{N-1}\varphi(T^kx)$. 
 We are interested in the almost sure behavior of $\limsup_{N\to\infty}\frac{\varphi\circ T^N}{S_N(\varphi)}$. This term is also referred to as \emph{extravagance}. If $\varphi\in \mathcal{L}^1(\mu)$, then it immediately follows that $\limsup_{N\to\infty}\frac{\varphi\circ T^N}{S_N(\varphi)}=0$
almost everywhere (a.e.). However, from now on we will assume that $\varphi\notin \mathcal{L}^1(\mu)$. If $(\varphi\circ T^{n-1})_{n\in\mathbb{N}}$ is a sequence of independent, identically distributed random variables (i.i.d.), then we always have 
$\limsup_{N\to\infty}\frac{\varphi\circ T^N}{S_N(\varphi)}=\infty$ a.e., see e.g.\ \cite{kesten}. For continued fraction mixing random variables the same statement can be deduced from \cite[Thm.\ 4 and following remark]{occu_times}. 
However, in the ergodic (non continued fraction mixing) context this is not always the case, see \cite[Ex.~a]{tanny} for an example of a Markov chain where $\lim_{N\to\infty}\frac{\varphi\circ T^N}{S_N(\varphi)}=0$ a.e.
Indeed, it is possible to construct examples 
for which $\limsup_{N\to\infty}\frac{\varphi\circ T^N}{S_N(\varphi)} =c$ a.e.\ holds, where $c$ can be any number between $0$ and $\infty$, see \cite{aaro_nak}. 

To see which dynamical systems may exhibit interesting extravagance behavior we note that 
$$\limsup_{N\to\infty}\frac{(\varphi\circ T^N)(x)}{S_N(\varphi)(x)}=c\in[0,\infty]$$ is equivalent to $$\limsup_{N\to\infty} \frac{\max\{\varphi(x),\ldots, (\varphi\circ T^{N-1})(x)\}}{S_N^{(1)}(\varphi)(x)}=c,$$ where the ($1$-)trimmed sum $S_N^{(1)}(\varphi)$ is defined as $S_N^{(1)}(\varphi)(x)=S_N(\varphi)(x)-\max\{\varphi(x),\ldots, (\varphi\circ T^{N-1})(x)\}$. 

The almost sure behavior of the $1$-trimmed sum of a suitable observable $\varphi$ usually having regularly varying tails with index $1$has been studied for a number of systems and for many of them a generalized strong law could be shown, i.e.\ the existence of a norming sequence $(d_N)$ such that
$\lim_{N\to\infty}\frac{S_N^{(1)}(\varphi)}{d_N}=1$ a.e., see e.g.\ \cite{mori} for i.i.d.\ random variables, \cite{diam_vaal, aaro_nak_trim} for continued fraction mixing dynamical systems and \cite{au_schi} for non-integrable observables over irrational rotations. By Aaronson's theorem it is clear that such a generalized strong law implies $\limsup_{N\to\infty} \frac{\varphi\circ T^N}{S_N(\varphi)}=\infty$ a.e.

In \cite{au_schi} it was shown that for the observable $\phi(x)=\frac{1}{x}-\frac{1}{1-x}$ the set of irrational rotations $\alpha$ for which $\lim_{N\to\infty}\frac{S_N^{(1)}(\phi)}{d_N}=1$ a.e.\ holds builds a full measure set. \footnote{Indeed, \cite{au_schi} uses as a roof function $\overline{\phi}:\T\to\R_{\geq0}$ with  $\overline{\phi}(x)=\frac{1}{x}$; however, the proofs can be easily adjusted to the observable $\phi$.}
However, it was also shown that for a set of exceptional $\alpha$ trimming by only the largest or any finite number of maximal large entries is not enough to obtain a strong law of large numbers.  
This makes irrational rotations with the observable $\phi$ an interesting object to study non-trivial extravagance behavior.

In the following Subsection \ref{subsec: extravag} we will state our results for the extravagance of $\phi$. 
Moreover, in Subsection \ref{subsec: hist behav} we will see that our results have immediate consequences for the extreme historic behavior of reparametrized linear flows on the torus.

\subsection{Extravagance for non-integrable observables over an irrational rotation}\label{subsec: extravag}
Let $\Omega=\mathbb{T}$ be the one-torus, $T(x)=R_{\alpha}(x)=x+\alpha\; \rm{mod}\, 1$, $\alpha\in(0,1)\backslash\mathbb{Q}$, and $\lambda$ be the Lebesgue measure on $\T$.  
As $\alpha$ is irrational, it has a unique continued fraction expansion 
\begin{align*}
 \alpha= [a_1 , a_2 , \ldots]=\cfrac{1}{a_1+\cfrac{1}{a_2+\cfrac{1}{\ddots}}}
\end{align*}
with $a_j \in \N$ for $j \geq 1$ and can be well approximated by finite iterates of the expansion, i.e.\ by the rational numbers
\begin{align*}
 \frac{p_n}{q_n}=\cfrac{1}{a_1+\cfrac{1}{a_2+\ldots+\cfrac{1}{a_{n}}}}
\end{align*}
with the convention $p_{-1}=q_0=1$ and $q_{-1} = 0$.
 It can be easily seen that $q_{n+1}=a_{n+1}q_n+q_{n-1}$.
The growth rate of the continued fraction exponents $q_n$ will be crucial to decide if the extravagance for the system $R_\alpha$ with the observable
$$\phi:\T\to\R_{\geq 0},\qquad \text{with}\qquad \phi(x)=\frac{1}{x}+\frac{1}{1-x}$$
equals $0$ or $\infty$.

\begin{theorem}\label{thm: main rot}
We have
\begin{align*}
  \limsup_{N\to\infty}\frac{\phi\circ R_{\alpha}^N}{S_N(\phi)}=\begin{cases} 0\\ \infty
  \end{cases}\text{ a.e.}
 \end{align*}
 according as 
 \begin{align}
  \sum_{n=1}^{\infty}W_n\begin{cases}<\infty \\ =\infty
\end{cases}
\label{eq: sum 1logqn}
 \end{align}
 with 
\begin{align*}
 W_n=W_n(\alpha)=\begin{cases}
      \frac{\log q_{n+1}-\log q_n}{\log q_n}&\text{if }q_{n+1}< 2q_n\log q_n\\
      \\
      \tfrac{\max\left\{1, \log\log q_n-\log\log \frac{q_{n+1}}{q_n\log q_n } \right\}}{\log q_n }&\text{if } q_{n+1}\geq 2q_n\log q_n.
     \end{cases}
\end{align*}
\end{theorem}

\begin{remark}\label{rem: Wn}
 We note that $W_n\log q_n$ can be arbitrarily small. For instance, if $a_n=1$, then 
 $$W_n\log q_n\leq \log q_{n+1}-\log q_n=\log(q_n+q_{n-1})-\log q_n= \log \left(1+\frac{q_{n-1}}{q_n}\right)$$
 with $\frac{q_{n-1}}{q_n}$ possibly arbitrarily small. 
 On the other hand, we generally have 
 \begin{align}
  W_n\ll \frac{\log\log q_n}{\log q_n}.
 \end{align}
 Loosely speaking, $W_n$ has its maximum in the situation that $q_{n+1}$ is close to $q_n\log q_n$ in which case $W_n$ equals (possibly up to a constant) $\frac{\log\log q_n}{\log q_n}$. 
\end{remark}

 In the next theorem we will be concerned with the size of the set of $\alpha$ with extravagance $0$. For this we denote by $\dim_{\mathrm{H}}(A)$ the Hausdorff dimension of a set $A$.

\begin{theorem}\label{prop: measzeroset}
We have that
 \begin{align}
  \lambda\Big(\alpha: \sum_{n=1}^{\infty} W_n(\alpha)<\infty \Big)=0\label{eq: Wn finite small}
 \end{align}
and 
\begin{align}
 \dim_{\mathrm{H}}\left(\alpha: \sum_{n=1}^{\infty} W_n(\alpha)<\infty\right)= \frac{1}{2}.\label{prop: HD dimension} 
\end{align}
\end{theorem}

It turns out that the order of the singularity is crucial for having an extravagance of $0$ as the following theorem shows.
In case we look at the observable $\widetilde{\phi}_{\gamma}(x)= x^{-\gamma}+(1-x)^{-\gamma}$, for $\gamma>1$, Theorem \ref{thm: main rot} simplifies to the following statement:
\begin{theorem}\label{thm: phi beta}
 For every $\alpha\in[0,1]\backslash\mathbb{Q}$, we have
\begin{align*}
  \limsup_{N\to\infty}\frac{\widetilde{\phi}_{\gamma}\circ R_{\alpha}^N}{S_N(\widetilde{\phi}_{\gamma})}=\infty\text{ a.e.}
 \end{align*}
\end{theorem}

\subsection{Historic behavior for reparametrizations of linear flows on the torus.}\label{subsec: hist behav}
Our main result or rather a special case of it has a consequence on (extreme) historic behavior for reparametrizations of linear flows on $\T^2$, strenghtening the results from \cite{An-Gu} which we now recall. Let $(L_t)$ be a linear flow on $\T^2$ generated by the constant vector field $X_0=(1,\alpha)$. One then considers a {\em reparametrization} of the flow $(L_t)$ given by a smooth function $\varphi:\T^2\to R_{\geq 0}$. More precisely let $\varphi$ be a $C^3$ non-negative function on $\T^2$ which vanishes at exactly two points ${\bf p}$ and ${\bf q}$. Moreover as in \cite{An-Gu} we consider the {\em non-degenerate case}, i.e.\ we assume that the Hessians $D^2\varphi ({\bf p})$, $D^2\varphi ({\bf q})$ are positive definite matrices. We then consider the flow $(T_t)$ which is generated by the vector field $X=X_\varphi=\varphi X_0$. Then it follows that $(T_t)$ has exactly two stopping points at ${\bf p}$ and ${\bf q}$. It follows moreover (see e.g.\ \cite{An-Gu}) that the only invariant probability measures for $(T_t)$ are convex combinations of $\delta_{{\bf p}}$ and $\delta_{{\bf q}}$.  We recall the following definition:
\begin{definition}\label{def:hist} For ${\bf x}\in \T^2$ consider the following {\em empirical measures}:
$$
\mu_t({\bf x})(f)=\frac{1}{t}\int_0^t f(T_s{\bf x})\, \mathrm{d}s.
$$
We say that the flow $(T_t)$ has:
\begin{enumerate}
\item[(i)] {\em physical measure} if for a.e.\ ${\bf x}\in \T^2$, $\lim_{t\to \infty} \mu_t({\bf x})= \mu_\infty$, for some $\mu_\infty$ being a convex combination of $\delta_{{\bf p}}$ and $\delta_{{\bf q}}$;
\item[(ii)] {\em historic behavior}  if for a.e.\ ${\bf x}\in \T^2$,  $\lim_{t\to \infty} \mu_t({\bf x})$ does not exist; 
\item[(iii)] {\em extreme historic behavior}  
if for a.e.\ ${\bf x}\in \T^2$, the accumulation points of the set $\{\mu_t({\bf x})\}_{t>0}$ are exactly the set of $(T_t)$-invariant probability measures.
\end{enumerate}
\end{definition}
The main result in \cite{An-Gu} is as follows: 
\begin{theorem}[{\cite[Theorem A]{An-Gu}}] There are subsets $\mathcal{F}, \mathcal{R},\mathcal{D}\subset  \T\times \T^2\times \T^2$ with the following characteristics: $\mathcal{F}$ is of full Lebesgue measure, $\mathcal{R}$ is a dense $G_\delta$ set, and $\mathcal{D}$ is dense (but not $G_\delta$), such that for any $(\alpha,{\bf p},{\bf q}) \in \T\times \T^2\times \T^2$ and the corresponding  reparametrized linear flow $(T_t)$ the following holds:
\begin{itemize}
\item $(T_t)$ has a unique physical measure if $(\alpha,{\bf p},{\bf q})\in \mathcal{D}$;
\item $(T_t)$ has historic behavior if $(\alpha,{\bf p},{\bf q})\in \mathcal{F}$;
\item $(T_t)$ has an extreme historic behavior if $(\alpha,{\bf p},{\bf q})\in \mathcal{R}$.
\end{itemize}
\end{theorem}
In fact the authors in \cite{An-Gu} conjecture that it should be possible to show that there is a full measure set $\mathcal{F}'\subset \T\times \T^2\times \T^2$ such that for any  $(\alpha,{\bf p},{\bf q})\in \mathcal{F}'$  the flow has extreme historic behavior. 
This is what we confirm:
\begin{theorem}\label{thm:conf} There exists a full measure set $\mathcal{F}'\subset \T\times \T^2\times \T^2$ such that for every $(\alpha,{\bf p},{\bf q})\in \mathcal{F}'$ the corresponding flow $(T_t)$ has extreme historic behavior.
\end{theorem}
 As shown in Section 2 and Section 3 (in particular Proposition 3.4.) in \cite{An-Gu}, 
we have the following criterion for extreme historic behavior:
\begin{proposition} 
Let $\Theta_N^\beta(x)=\frac{S_N(\phi)(x)}{S_N(\phi)(x-\beta)}$ where  $\phi(x)=\frac{1}{x}+\frac{1}{1-x}$. If
\begin{equation}\label{eq:aex}
\limsup_{N\to\infty}\Theta_N^\beta(x)=\infty\;\;\;\text{ and }\;\;\; \liminf_{N\to\infty}\Theta_N^\beta(x)=0\quad\text{ for a.e. } x\in \T,
\end{equation}
then the flow $(T_t)$ given by the parameters $(\alpha, {\bf p},{\bf p}+(\beta,s))$, $s\in [0,1)$, has extreme historic behavior. 
\end{proposition}

Using the above proposition we will show that the set of $(\alpha,\beta)$ for which the corresponding flow has extreme historic behavior is of full measure: 
\begin{proposition}\label{prop: full measure set} There exists a full measure set $Z\subset \T^2$ such that for every $(\alpha,\beta)\in Z$ we have that \eqref{eq:aex} holds.
\end{proposition}

\section{Proofs}
Before starting with the proofs of the theorems, we first introduce some additional notation. Denote
\begin{align*}
 w_n=|\alpha q_n|\cdot q_{n+1}
\end{align*}
which implies $w_n\in (\frac{1}{2}, 1)$. Moreover, a number of times we will use the Denjoy-Koksma inequality which states the following: for $f$ a function of bounded variation and $n \geq 1$ it holds for all $x\in \T$ that
\begin{align*}
 \left| S_{q_n} (f )(x) - q_n \int f \mathrm{d}\lambda \right|\leq  \Var(f ),
\end{align*}
where $\Var(f)$ denotes the total variation of $f$.

Finally, we introduce the following two conventions: 
\begin{itemize}
 \item $u_n\asymp v_n$: There exists $C>1$ such that $C^{-1}u_n\leq v_n\leq Cu_n$, for all $n\in\mathbb{N}$. 

\item $u_n\lesssim v_n$: We have $\limsup_{n\to\infty}\frac{u_n}{v_n}<\infty$.
\end{itemize}

\subsection{Proof of Theorem \ref{thm: main rot}}
\begin{proof}[Proof of Theorem \ref{thm: main rot}]
The proof will be done in the following two steps:
\begin{enumerate}
 \item\label{en: 1} If \eqref{eq: sum 1logqn} is finite, then for all $\epsilon>0$, we have 
\begin{align}\label{eq: limsup 0}
  \limsup_{N\to\infty}\frac{\phi\circ R_{\alpha}^N}{S_N(\phi)}<\epsilon \text{ a.e.}
 \end{align}
  \item\label{en: 2} If \eqref{eq: sum 1logqn} is infinite, then for all $c>0$, we have 
\begin{align}\label{eq: limsup 1}
  \limsup_{N\to\infty}\frac{\phi\circ R_{\alpha}^N}{S_N(\phi)}>c \text{ a.e.}
 \end{align}
\end{enumerate}
For the following let $x_{\min,k}=\min\{d(0,x),\ldots, d(0,R_{\alpha}^{k-1}(x))\}$, where $d$ denotes the distance on the torus.   Moreover, during the proof we will write $S_n$ instead of $S_n(\phi)$ and $S_n^{(1)}$ instead of $S_n^{(1)}(\phi)$ unless there is a possibility of confusion.
The following lemma can be deduced from the Denjoy-Koksma inequality:
\begin{lemma}\label{lem: DK adapted}
 For all $x\in\mathbb{T}$, we have 
 \begin{align*}
  \left|S_{q_n}(x)- 2 q_n\log q_n - \phi(x_{\min,q_n})\right|\leq 2 q_n.
 \end{align*}
\end{lemma}

This lemma allows us to give the following precise estimates of the ergodic sums. 

\begin{lemma}\label{lem: DK adapted 1}
 Let $k\in [j q_n, (j+1) q_n]$ and $j\in\mathbb{N}$ being such that $(j+1) q_n< q_{n+1}$. 
Then we have 
 \begin{align*}
  2j q_n\log q_n-2jq_n\leq S_k(x)-\phi(x_{\min,k})\leq 2(j+1) q_n\log q_n + 2(j+1) q_n + 4q_{n+1}(2+ \log j).
 \end{align*}

 If we additionally assume that 
 \begin{align}
 x_{\min, k}\in \left(-\frac{1}{4q_{n+1}}, \frac{1}{4q_{n+1}}\right), \label{eq: xmin close to zero}
 \end{align}
 then 
\begin{align}
S_k(x)-\phi(x_{\min,k})\geq  2j q_n\log q_n - 2jq_n +\frac{4}{5}q_{n+1}\max\left\{0,\log (j-2)\right\}.\label{eq: second lower bound}
\end{align}
\end{lemma}
\begin{proof}
 First, we denote 
 $x_{\min,A}=\min\{d(0,R_{\alpha}^{ \ell-1}(x)),{ \ell}\in A\cap\mathbb{N}\}$.

 In order to estimate the first inequality, we notice that by Lemma \ref{lem: DK adapted} we have 
 \begin{align}
  S_k(x)\geq S_{jq_n}(x)\geq 2jq_n\log q_n-2jq_n+\sum_{\ell=1}^j \phi(x_{\min,((\ell-1)q_n, \ell q_n]})\label{eq: Sk estim in lemma}
 \end{align}
 and the first inequality follows immediately.

 In order to prove the upper bound,
similarly as above, we obtain by Lemma \ref{lem: DK adapted} that 
\begin{align*}
   S_k(x)\leq S_{(j+1)q_n}(x)\leq 2(j+1)q_n\log q_n+2(j+1)q_n+\sum_{\ell=1}^{j+1} \phi(x_{\min,((\ell-1)q_n, \ell q_n]}).
\end{align*}
First, we note that $\phi(x_{\min,k})\leq \phi(0)=\infty$.
Remembering that $|R_{\alpha}^{q_n}(0)|\geq \frac{1}{2q_{n+1}}$, 
we have that in the interval $[-\frac{1}{2q_{n+1}}, \frac{1}{2q_{n+1}}]$ there are at most two interates of $R_{q_n\alpha}^{\ell}$ with the second closest not closer than $\frac{1}{4q_{n+1}}$. Moreover, as the $r$th closest iterate on each side (right and left of zero) has a distance of at least $\frac{r-1}{2q_{n+1}}$ we can estimate 
\begin{align*}
 \sum_{\ell=1}^{j+1} \phi(x_{\min,((\ell-1)q_n, \ell q_n]})-\phi(x_{\min, k})
 &\leq \phi\Big(\frac{1}{4q_{n+1}}\Big)+
 2\sum_{\ell=1}^{\lceil (j-1)/2\rceil} \phi\Big(\frac{\ell}{2q_{n+1}}\Big)
 \leq 4q_{n+1}(2+\log j),
\end{align*}
 finishing the proof of the upper bound.
 
 Finally, we prove \eqref{eq: second lower bound}. Since we were assuming \eqref{eq: xmin close to zero}, it follows that there exists $\ell\in \{1,\ldots, j\}$ such that $d(0, x_{\min, ((\ell-1)jq_n, \ell q_n]})\leq \frac{5}{4q_{n+1}}$. This follows from the fact that $|R_{\alpha}^{q_n}(0)|\leq \frac{1}{q_{n+1}}$. 
By an inductive argument, it follows that within the first $jq_n$ iterates for each $\ell\in \{1,\ldots, j\}$ there are at least $\ell$ points which have a distance at least $\frac{1+4 (\ell-1)}{4q_{n+1}}$ to $0$. 
Hence,
\begin{align*}
 \sum_{\ell=1}^j \phi(x_{\min,((\ell-1)q_n, \ell q_n]})- \phi(x_{\min, k})
 \geq \sum_{\ell=2}^j \phi \left(\frac{1+4 (\ell-1)}{4q_{n+1}}\right)
 \geq \frac{4}{5}\max\{\log (j-2),0\} q_{n+1}.
\end{align*}
This together with \eqref{eq: Sk estim in lemma} gives the lower bound.
\end{proof}

We start the proof of Theorem \ref{thm: main rot} by proving the first case \eqref{en: 1}. Assume the sum in \eqref{eq: sum 1logqn} is finite. 
The idea is to construct a superset of those points $x\in [0,1)$ for which $\phi\circ R_{\alpha}^n(x)\geq \epsilon S_n(x)$ holds infinitely often and show then that under the condition that \eqref{eq: sum 1logqn} is finite this set has zero measure.

First, we define 
\begin{align*}
 \mathcal{N}_{\epsilon}:=\{n\in\mathbb{N}: q_{n+1}< \epsilon q_n\log q_n\}.
\end{align*}
Let us first assume that $n\in \mathcal{N}_{\epsilon}$.
In this case, for $j\leq \lfloor \frac{q_{n+1}}{q_n}\rfloor-1$ we define
 \begin{align*}
  \overline{A}_{n,j}
  &=\bigcup_{i=0}^{q_{n}-1} 
  \left[ R_{\alpha}^{-i-jq_n}(0)-\frac{1}{\epsilon j q_n\log q_n}\, ,\, R_{\alpha}^{-i-jq_n}(0) + \frac{1}{\epsilon j q_n\log q_n}\right]
 \end{align*}
 and for the case $j=\lfloor \frac{q_{n+1}}{q_n}\rfloor$ we distinguish
 \begin{align*}
  \overline{A}_{n,\lfloor q_{n+1}/q_n\rfloor}
  &=\bigcup_{i=\lfloor q_{n+1}/q_n\rfloor q_n}^{q_{n+1}-1} 
  \left[ R_{\alpha}^{-i}(0)-\frac{1}{\epsilon (\lfloor \frac{q_{n+1}}{q_n}\rfloor-1) q_n\log q_n}\,,\right.\\
  &\qquad\qquad\qquad\qquad\left. R_{\alpha}^{-i}(0) + \frac{1}{\epsilon (\lfloor \frac{q_{n+1}}{q_n}\rfloor-1) q_n\log q_n}\right] &\text{if } \left\lfloor \frac{q_{n+1}}{q_n}\right\rfloor>1\\
  \overline{A}_{n,1}
  &=\bigcup_{i= q_n}^{q_{n+1}-1} 
  \left[ R_{\alpha}^{-i}(0)-\frac{1}{\epsilon  q_n\log q_n}\,,\, R_{\alpha}^{-i}(0) + \frac{1}{\epsilon q_n\log q_n}\right] &\text{if }\left\lfloor \frac{q_{n+1}}{q_n}\right\rfloor=1.
  \end{align*}
Moreover, we define 
\begin{align}
 \overline{B}_n= \bigcup_{j=1}^{\lfloor q_{n+1}/q_n\rfloor } \overline{A}_{n,j}.\label{eq: def sup B}
\end{align}
In the next steps, we prove that for each $j\in\{0,\ldots, \lfloor \frac{q_{n+1}}{q_n}\rfloor \}$ we have 
 \begin{align}
  C_{n,j}:=\left\{x: \exists k\in  [j q_n,\min \{(j+1)q_n, q_{n+1}\}): \phi\circ R_{\alpha}^k(x)> \epsilon S_{k}(x) \right\} \subset \overline{A}_{n,j}.\label{eq: Cnj subset Anj}
 \end{align}
Note that for $k\geq jq_n-1$, by Lemma \ref{lem: DK adapted} we have 
 \begin{align*}
  S_k(x)\geq  S_{jq_n}(x)-\phi(x_{\min,jq_n})\geq  2jq_n\log q_n-2jq_n\geq  (2-\epsilon) jq_n\log q_n, 
 \end{align*}
for $n$ sufficiently large.

On the other hand, for $k\in  [jq_n, (j+1)q_n)$ and $n$ sufficiently large, say for $n\geq U$, the statement $\phi\circ R_{\alpha}^k(x)>\epsilon S_k(x)$ implies 
$\phi\circ R_{\alpha}^k(x)>\epsilon (2-\epsilon) jq_n\log q_n$. This implies 
$R_{\alpha}^k(x)\in [-(\epsilon  jq_n\log  q_n)^{-1}, (\epsilon  jq_n\log q_n)^{-1}]$ for $n$ sufficiently large which is equivalent to
$ x\in \overline{A}_{n,j}$
and \eqref{eq: Cnj subset Anj} follows. 
\newline 

Next, assume that $n\notin \mathcal{N}_{\epsilon}$ and define 
\begin{align*}
 J_{n,\epsilon}= \left\lceil \frac{ 5 q_{n+1}}{\epsilon q_n \log q_n}\right\rceil+1.
\end{align*}
Note that for $n$ sufficiently large, $\frac{q_{n+1}}{q_n}> 2J_{n,\epsilon}$. 
Let $U_\epsilon:=\min\{n\geq U: \frac{q_{k+1}}{q_k}> 2J_{k,\epsilon}\text{ for all }k\geq n\}$.

Moreover, define 
\begin{align*}
 \overline{A}_{n,1}'=\begin{cases}
                      \bigcup_{i=0}^{q_n-1}\left[R_{\alpha}^{-i-J_{n,\epsilon}}(0)\,,\, R_{\alpha}^{-i}(0)+\frac{1}{\epsilon q_n\log q_n}\right] 
                      &\text{if }\alpha-\frac{p_n}{q_n}>0\\
                      &\\
                      \bigcup_{i=0}^{q_n-1}\left[R_{\alpha}^{-i}(0)-\frac{1}{\epsilon q_n\log q_n}\,,\, R_{\alpha}^{-i-J_{n,\epsilon}}(0)\right] 
                      &\text{if }\alpha-\frac{p_n}{q_n}<0. 
                     \end{cases}
\end{align*}
Furthermore, for $j\geq J_{n,\epsilon}$, we define 
\begin{align*}
 v_{n,j}=\epsilon j q_n\log q_n+\frac{\epsilon}{2} q_{n+1} \max\{0,\log(j-2)\}
\end{align*}
and
\begin{align*}
 \overline{A}_{n,j}'
 =\bigcup_{i=0}^{q_{n}-1} 
  \left[ R_{\alpha}^{-i-jq_n}(0)-\frac{1}{v_{n,j}}\, ,\, R_{\alpha}^{-i-jq_n}(0) + \frac{1}{v_{n,j}}\right]
\end{align*}
and finally
\begin{align}
 \overline{B}_n=\overline{A}_{n,1}'\cup \bigcup_{j=J_{n,\epsilon}}^{\lfloor q_{n+1}/q_n\rfloor} \overline{A}_{n,j}'.\label{eq: overline B 2nd Def}
\end{align}
First, for $n\in \mathcal{N}_{\epsilon}^c\cap\mathbb{N}_{\geq U_\epsilon}$ we will show that 
$\bigcup_{j=1}^{J_{n,\epsilon}-1} C_{n,j}\subset \overline{A}_{n,1}'$.  

Similarly as in the argumentation above, for $k\in [q_n,J_{n,\epsilon}q_n)$, the occurrence of the event $\phi\circ R_{\alpha}^k(x)>\epsilon S_k(x)$ implies 
\begin{align*}
 x&\in \bigcup_{j=1}^{J_{n,\epsilon}-1}\bigcup_{i=0}^{q_n-1} \left[R_{\alpha}^{-i-jq_n}(0)-\frac{1}{\epsilon  jq_n\log q_n}\,,\, R_{\alpha}^{-i-jq_n}(0)+ \frac{1}{\epsilon  jq_n\log q_n}\right]=:\bigcup_{j=1}^{J_{n,\epsilon}-1} \bigcup_{i=0}^{q_n-1}\Gamma_{n,i,j},
\end{align*}
for $n$ sufficiently large.
For the moment, assume $\alpha-\frac{p_n}{q_n}>0$. The other case follows analogously with a symmetry at $0$. 
We note that $\Gamma_{n,i,j}$ can also be written as
\begin{align*}
\Gamma_{n,i,j} & =\left[R_{\alpha}^{-i}(0)-\frac{w_n j}{q_{n+1}}-\frac{1}{\epsilon  jq_n\log q_n}\,,\, R_{\alpha}^{-i}(0)-\frac{w_n j}{q_{n+1}}+ \frac{1}{\epsilon  jq_n\log q_n}\right].
\end{align*}
Moreover, we note that 
\begin{align*}
 \frac{w_n j}{q_{n+1}}+\frac{1}{\epsilon  jq_n\log q_n}\leq \frac{w_n J_{n,\epsilon}}{q_{n+1}}\quad\text{ iff }\quad
 \frac{q_{n+1}}{\epsilon  w_n q_n\log q_n}\leq j(J_{n,\epsilon}-j).
\end{align*}
As the polynomial on the right hand side has its minimum at the endpoints, the inequality is true for all $j\in\{1,\ldots, J_{n,\epsilon}-1\}$ if 
\begin{align*}
 \frac{q_{n+1}}{\epsilon  w_n q_n\log q_n}\leq J_{n,\epsilon}-1
\end{align*}
which is true by the definition of $J_{n,\epsilon}$ and the fact that $w_n\in (\frac{1}{2}, 1)$.
Hence, 
\begin{align*}
 \bigcup_{j=1}^{J_{n,\epsilon}-1} \Gamma_{n,i,j}
 \subset \left[ R_{\alpha}^{-i}(0)-\frac{w_nJ_{n,\epsilon}}{q_{n+1}}\, , \, R_{\alpha}^{-i}(0)+\frac{1}{\epsilon q_n \log q_n}\right]
\end{align*}
implying  
$\bigcup_{j=1}^{J_{n,\epsilon}-1}\bigcup_{i=0}^{q_n-1} \Gamma_{n,i,j}\subset \overline{A}_{n,1}'$ and thus 
\begin{equation}
 \bigcup_{j=1}^{J_{n,\epsilon}-1} C_{n,j}\subset \overline{A}_{n,1}'.\label{eq: C subset A'}
\end{equation}

For $n\in \mathcal{N}_{\epsilon}^c\cap\mathbb{N}_{\geq U_\epsilon}$ and $j\geq J_{n,\epsilon}$, we need the more precise estimate of $\overline{A}_{n,k}'$ instead of $\overline{A}_{n,k}$ as in the case $n\in \mathcal{N}_{\epsilon}$. 
Noting that by Lemma \ref{lem: DK adapted 1}
we have for  $k\in [jq_n, (j+1)q_n)$ and $n$ sufficiently large, the statement $\phi\circ R_{\alpha}^k(x)>\epsilon S_k(x)$ implies 
 $\phi\circ R_{\alpha}^k(x)>\epsilon  jq_n\log q_n$. In case $j\leq J_{n,\epsilon}$, this implies $ x\in \overline{A}_{n,1}'$. On the other hand, if $j>J_{n,\epsilon}$, then $\epsilon S_k(x)> \epsilon J_{n,\epsilon}q_n\log q_n>5q_{n+1}$. Then $\phi\circ R_{\alpha}^k(x)> \epsilon S_k(x)$ implies that \eqref{eq: xmin close to zero} is fulfilled and we can apply the stronger upper bound \eqref{eq: second lower bound}.
 In particular, we have
$\phi\circ R_{\alpha}^k(x)>\epsilon  jq_n\log q_n+ \frac{\epsilon}{2} q_{n+1}\max\{\log(j-2), 0\}$ for $n$ sufficiently large.
This implies 
$R_{\alpha}^k(x)\in [-\frac{1}{v_{n,j}}, \frac{1}{v_{n,j}}]$ which is equivalent to
$ x\in \overline{A}_{n,j}'$. 
Thus, we obtain  
\begin{equation}
 C_{n,j}\subset \overline{A}_{n,j}'\quad\text{for }j\geq J_{n,\epsilon}\label{eq: C subset A'2}
\end{equation}
and hence
$\bigcup_{j=1}^{\lfloor q_{n+1}/q_n\rfloor} C_{n,j}\subset \overline{B}_n$.
\newline

In the next steps we will estimate $\lambda(\overline{B}_n)$. 
First assume $n\in\mathcal{N}_{\epsilon}$ and $a_{n+1}\neq 1$. In particular, this implies that $\lfloor \frac{q_{n+1}}{q_n}\rfloor \geq 2$. In this case we have 
\begin{align}
\lambda(\overline{B}_n)
 &=\sum_{j=1}^{\lfloor q_{n+1}/q_n\rfloor} \lambda(\overline{A}_{n,j})
 \leq\sum_{j=1}^{\lfloor q_{n+1}/q_n\rfloor}\frac{2}{\epsilon j \log q_n}
 \ll \frac{\log q_{n+1}-\log q_n}{\log q_n}.\label{eq: lambda upper B a}
\end{align}
If $a_{n+1}=1$ (and thus automatically $n\in\mathcal{N}_{\epsilon}$), we have $q_{n+1}\leq 2q_n$. Thus, we can estimate 
\begin{align}
 \lambda(\overline{B}_n)
 =\sum_{j=1}^{\lfloor q_{n+1}/q_n\rfloor} \lambda(\overline{A}_{n,j})
 =\lambda(\overline{A}_{n,1})
 =\frac{q_{n+1}-q_n}{\epsilon q_n\log q_n}
 \ll \frac{q_{n+1}-q_n}{q_{n}}\frac{1}{\log q_n}
 \leq \frac{\log q_{n+1}-\log q_n}{\log q_n}.\label{eq: lambda upper B a1}
\end{align}

Next consider $n\in\mathcal{N}_{\epsilon}^c$.
There exists $U_\epsilon'\geq U_{\epsilon}$ such that for $n\in \mathcal{N}_{\epsilon}^c\cap\mathbb{N}_{\geq U_{\epsilon}'}$ we have
\begin{align}
 \lambda(\overline{B}_n)
 &=\lambda(\overline{A}_{n,1}')+\sum_{j=J_{n,\epsilon}}^{\lfloor q_{n+1}/q_n\rfloor} \lambda(\overline{A}_{n,j}')
 \leq \frac{10}{\epsilon \log q_n}+  \sum_{j=J_{n,\epsilon}}^{\lfloor q_{n+1}/q_n\rfloor} \frac{2q_n}{v_{n,j}}.\label{eq: lambda upper B}
\end{align}
First, set $L_{\epsilon}=2\left(\lceil \frac{2}{\epsilon}\rceil +1\right)$ 
and consider the case $n\in (\mathcal{N}_{L_{\epsilon}}\backslash \mathcal{N}_{\epsilon})\cap\mathbb{N}_{\geq U_{\epsilon}'}$. 
Then we have 
\begin{align}
 \sum_{j=J_{n,\epsilon}}^{\lfloor q_{n+1}/q_n\rfloor} \frac{2q_n}{v_{n,j}}
 &\leq \sum_{j=J_{n,\epsilon}}^{\lfloor L_{\epsilon} \log q_n\rfloor}
 \frac{2}{\epsilon j \log q_n+\frac{\epsilon^2}{2} \log q_n \max\{0,\log(j-2)\}}\notag\\
 &\leq \sum_{j=J_{n,\epsilon}}^{\lfloor L_{\epsilon} \log q_n\rfloor}
 \frac{2}{\epsilon j \log q_n}
 \ll \frac{\log (\lfloor L_{\epsilon} \log q_n\rfloor)}{\log q_n}
 \ll \frac{\log\log q_n}{\log q_n}.\label{eq: lambda upper B b}
\end{align}

Next, we consider the case $n\in \mathcal{N}_{L_{\epsilon}}^c\cap\mathbb{N}_{\geq U_{\epsilon}'}$.
Denoting 
\begin{align*}
 K_n=\frac{q_{n+1}}{q_n\log q_n}\log\left( \frac{q_{n+1}}{q_n\log q_n}\right),
\end{align*}
the last sum in \eqref{eq: lambda upper B} can be estimated as follows: 
\begin{align}
 \sum_{j=J_{n,\epsilon}}^{\lfloor q_{n+1}/q_n\rfloor} \frac{2q_n}{v_{n,j}}
 &\leq \sum_{j=J_{n,\epsilon}}^{\lfloor q_{n+1}/q_n\rfloor} \frac{2}{\epsilon j \log q_n+\frac{\epsilon q_{n+1}}{2 q_n}  \max\{0,\log(j-2)\}}\notag\\
 &\leq  \frac{2}{\epsilon J_{n,\epsilon} \log q_n}+\sum_{j=J_{n,\epsilon}+1}^{\min\{K_n, \lfloor q_{n+1}/q_n\rfloor\}} \frac{4 q_n}{\epsilon q_{n+1}  \max\{0,\log(j-2)\}}+\sum_{j=K_n+1}^{\lfloor q_{n+1}/q_n\rfloor}
 \frac{2}{\epsilon j \log q_n}\label{eq: lambda upper B1}
\end{align}
with the convention that $\sum_{j=a}^b$ is the empty sum if $b<a$. 
Note that by assuming that $n\in\mathcal{N}_{L_{\epsilon}}^c$ (and assuming that $\epsilon$ is sufficiently small), we have that $K_n\geq 2 J_{n,\epsilon}$ and the first sum can not be empty. 
Moreover, since $n\in\mathcal{N}_{\epsilon}^c$, we have $J_{n,\epsilon}\geq 6$ and thus 
\begin{align}
 \frac{2}{\epsilon J_{n,\epsilon} \log q_n}\ll \frac{1}{\log q_n}. \label{eq: lambda upper B1 first summand}
\end{align}
In particular, it also implies that $J_{n,\epsilon}-2>1$.
Hence, for $n$ sufficiently large, say for $n\geq U_{\epsilon}''$ the first of the above sums can be estimated as 
\begin{align*}
 \MoveEqLeft\sum_{j=J_{n,\epsilon}+1}^{\min\{K_n, \lfloor q_{n+1}/q_n\rfloor\}} \frac{4 q_n}{\epsilon q_{n+1}  \max\{0,\log(j-2)\}}\notag\\
 &\leq \frac{4 q_n}{\epsilon q_{n+1}}{\int_{J_{n,\epsilon}-2}^{K_n}\left(
 \frac{1}{\log x}\right)\,\mathrm{d}x}\notag\\
  &\leq \frac{8 q_n}{\epsilon q_{n+1}}\left( \frac{K_n}{\log K_n}- 
 \frac{J_{n,\epsilon}}{2\log J_{n,\epsilon}}\right)\notag\\
 &\leq \frac{8 q_n}{\epsilon q_{n+1}}\left( \frac{q_{n+1}\log\left( \frac{q_{n+1}}{q_n\log q_n}\right)}{q_n\log q_n \log \left( \frac{q_{n+1}\log\left( \frac{q_{n+1}}{q_n\log q_n}\right)}{q_n\log q_n }\right)}- 
 \frac{q_{n+1}}{q_n \log q_n \log \left( \frac{q_{n+1}}{q_n\log q_n}\right)}\right)\notag\\
  &=\frac{8 }{\epsilon \log q_n}\left( \frac{\log\left( \frac{q_{n+1}}{q_n\log q_n}\right)}{ \log \left( \frac{q_{n+1}\log\left( \frac{q_{n+1}}{q_n\log q_n}\right)}{q_n\log q_n }\right)}
  - 
 \frac{1}{  \log \left( \frac{q_{n+1}}{q_n\log q_n}\right)}\right).
 \end{align*}
Noticing that 
\begin{align*}
 \frac{\log x}{ \log \left( x \log x\right)}
  - 
 \frac{1}{  \log x}>0\,\, \text{ for }x>0\quad\text{ and } \quad\frac{\log x}{ \log \left( x \log x\right)}
  - 
 \frac{1}{  \log x}\sim \frac{1}{\log x}\,\,\text{ for }x\to\infty,
\end{align*}
we can conclude 
 \begin{align}
\sum_{j=J_{n,\epsilon}+1}^{\min\{K_n, \lfloor q_{n+1}/q_n\rfloor\}} \frac{2 q_n}{\epsilon q_{n+1}  \max\{0,\log(j-2)\}}
 &\ll \frac{1}{\log q_n}.\label{eq: lambda upper B11}
\end{align}
Next, we estimate the second sum in \eqref{eq: lambda upper B1}.
If $\frac{q_{n+1}}{q_n}\leq 2 K_n$, we have 
\begin{align}
 \sum_{j=K_n+1}^{\lfloor q_{n+1}/q_n\rfloor}
 \frac{2}{\epsilon j \log q_n}
 &\leq \left(\frac{q_{n+1}}{q_n}-K_n\right)\frac{2}{\epsilon K_n \log q_n}
 \leq \frac{2 K_n}{\epsilon K_n \log q_n}
 \ll \frac{1}{\log q_n}. \label{eq: lambda upper B12a}
\end{align}
So, for the following, let us assume $\frac{q_{n+1}}{q_n}\geq 2 K_n$. In this case we have 
\begin{align}
 \sum_{j=K_n+1}^{\lfloor q_{n+1}/q_n\rfloor}
 \frac{1}{\epsilon j \log q_n}
 &\ll \frac{1}{\log q_n}\left( \log \left(\frac{ q_{n+1}}{q_n}\right)-
 \log \left( \frac{q_{n+1}}{q_n\log q_n}\log\left( \frac{q_{n+1}}{q_n\log q_n}\right)\right)\right)\notag\\
 &\leq \frac{\log\log q_n -\log\log \left(\frac{q_{n+1}}{q_n\log q_n}\right) }{\log q_n}.\label{eq: lambda upper B12b}
\end{align}
Combining \eqref{eq: lambda upper B}, \eqref{eq: lambda upper B1}, \eqref{eq: lambda upper B1 first summand}, \eqref{eq: lambda upper B11}, \eqref{eq: lambda upper B12a}, and \eqref{eq: lambda upper B12b}, 
we have for $n\in \mathcal{N}_{L_{\epsilon}}^c$ that
\begin{align}
 \lambda\left(\overline{B}_n\right)
 &\ll \frac{\max\left\{1\, ,\, \log\log q_n -\log\log \left(\frac{q_{n+1}}{q_n\log q_n}\right) \right\}}{\log q_n}.\label{eq: lambda upper B c}
\end{align}

Hence, combining \eqref{eq: lambda upper B a}, \eqref{eq: lambda upper B a1}, \eqref{eq: lambda upper B b}, and \eqref{eq: lambda upper B c}, there exists $U_{\epsilon}''$ such that 
\begin{align*}
 \sum_{n=U_{\epsilon}''}^{\infty} \lambda\left(\overline{B}_n\right)
 &\leq \sum_{n\in \mathcal{N}_{\epsilon}\cap \mathbb{N}_{U_{\epsilon}''}} \frac{\log q_{n+1}-\log q_n}{\log q_n}
 +\sum_{n\in (\mathcal{N}_{L_\epsilon}\backslash \mathcal{N}_{\epsilon})\cap \mathbb{N}_{U_{\epsilon}''}} \frac{\log\log q_n}{\log q_n}\\
 &\qquad +\sum_{n\in \mathcal{N}_{L_\epsilon}^c\cap \mathbb{N}_{U_{\epsilon}''}}\frac{\max\left\{1\, ,\, \log\log q_n -\log\log \left(\frac{q_{n+1}}{q_n\log q_n}\right) \right\}}{\log q_n}.
\end{align*}
Moreover, for $n\in \mathcal{N}_2\backslash \mathcal{N}_{\epsilon}$ we have 
\begin{align*}
 \log\log q_n
 &= \log q_n+\log\log q_n-\log q_n
 \leq \log q_{n+1}-\log q_n+\log \epsilon \ll \log q_{n+1}-\log q_n.
\end{align*}
On the other hand, for $n\in \mathcal{N}_{L_{\epsilon}}\backslash \mathcal{N}_{2}$ we have  
\begin{align*}
 \log\log q_n &\ll \log\log q_n-\log\log L_\epsilon 
 \leq  \log\log q_n -\log\log \left(\frac{q_{n+1}}{q_n\log q_n}\right).
\end{align*}
Thus, there exists $U_{\epsilon}^{(iv)}$ such that setting $\mathbb{N}^{(iv)}=\mathbb{N}_{\geq U_{\epsilon}^{(iv)}}$ we have
\begin{align*}
 \sum_{n=U_{\epsilon}^{(iv)}}^{\infty} \lambda\left(\overline{B}_n\right)
 &\leq \sum_{n\in \mathcal{N}_2\cap \mathbb{N}^{(iv)}} \frac{\log q_{n+1}-\log q_n}{\log q_n}\\ 
 &\qquad +\sum_{n\in \mathcal{N}_{2}^c\cap\mathbb{N}^{(iv)}}\frac{\max\left\{1\, ,\, \log\log q_n -\log\log \left(\frac{q_{n+1}}{q_n\log q_n}\right) \right\}}{\log q_n}.
\end{align*}
Since we are assuming that \eqref{eq: sum 1logqn} converges, we obtain by the first Borel-Cantelli lemma that 
\begin{align*}
 \lambda \left( \bigcup_{j=1}^{\lfloor q_{n+1}/q_n\rfloor} C_{n,j}\text{ for infinitely many }n\in\mathbb{N}\right)=0. 
\end{align*}
By \eqref{eq: def sup B} and \eqref{eq: Cnj subset Anj} in case $n\in\mathcal{N}_{\epsilon}$ and by \eqref{eq: overline B 2nd Def}, \eqref{eq: C subset A'} and \eqref{eq: C subset A'2} in case $n\in \mathcal{N}_{\epsilon}^c$
this implies statement \eqref{en: 1}. 
\newline 

In the next steps, we will prove statement \eqref{en: 2}.
The idea is to construct a subset of those points for which \eqref{eq: limsup 1} holds infinitely often and show then that under the condition that \eqref{eq: sum 1logqn} is infinite this set has full measure. 

First, we define 
\begin{align}\label{vnj}
 \underline{v}_{n,j}= 3c(j+1) q_n\log q_n +  4cq_{n+1}(2+ \log j)
\end{align}
and 
 \begin{align*}
  \underline{A}_{n,j}
  &=\bigcup_{i=0}^{q_{n}-1} 
  \left[ R_{\alpha}^{-i-jq_n}(0)-\frac{1}{\underline{v}_{n,j}}\, ,\, R_{\alpha}^{-i-jq_n}(0) + \frac{1}{\underline{v}_{n,j}}\right]\quad \text{for }j\leq \Big\lfloor \frac{q_{n+1}}{q_n}\Big\rfloor -1,\\
  \underline{A}_{n,\lfloor q_{n+1}/q_n\rfloor}
  &=\bigcup_{i=\lfloor q_{n+1}/q_n\rfloor q_n}^{q_{n+1}-1} 
  \left[ R_{\alpha}^{-i}(0)-\frac{1}{\underline{v}_{n,\lfloor q_{n+1}/q_n\rfloor}}\, ,\, R_{\alpha}^{-i}(0) + \frac{1}{\underline{v}_{n,\lfloor q_{n+1}/q_n\rfloor}}\right].
\end{align*}

Furthermore, we define
\begin{align*}
 \underline{B}_n=\bigcup_{j=1}^{\lfloor q_{n+1}/q_n\rfloor}\underline{A}_{n,j}.
\end{align*}

The sets  $\{\underline{B}_n\}$ are crucial in what follows. Their properties will be summarized in Lemma \ref{lem:add} and Lemma \ref{lem:add2} and used below.

In the next steps, we prove that for each $j\in\{0,\ldots, \lfloor \frac{q_{n+1}}{q_n}\rfloor \}$ we have 
 \begin{align}
  \underline{C}_{n,j}:=\left\{x: \exists k\in  [j q_n,\min \{(j+1)q_n, q_{n+1}\}): \phi\circ R_{\alpha}^k(x)> c S_{k}(x) \right\} \supset \underline{A}_{n,j}.\label{eq: Cnj supset Anj}
 \end{align}

First note that the statement $ x\in \underline{A}_{n,j}$ is equivalent to the statement that there is a $ k\in  [j q_n,\min \{(j+1)q_n, q_{n+1}\})$ such that
$R_{\alpha}^k(x)\in [-\frac{1}{\underline{v}_{n,j}}, \frac{1}{\underline{v}_{n,j}}]$. This implies that there exists $k\in  [jq_n,\min \{(j+1)q_n, q_{n+1}\})$ such that 
$\phi\circ R_{\alpha}^k(x)>\underline{v}_{n,j}$.
However, for such $k$, we can conclude by Lemma \ref{lem: DK adapted 1}
that $S_{k+1}^{(1)}(x)\leq  \frac{\underline{v}_{n,j}}{c}$, 
for $n$ sufficiently large.
Hence, in order to show \eqref{eq: Cnj supset Anj} it only remains to show that in this case $S_k(x)\leq S_{k+1}^{(1)}(x)$ which is equivalent to saying that $\phi\circ R_{\alpha}^k(x)=\max\{\phi(x),\ldots, \phi\circ R_{\alpha}^k(x)\}$. However, this has to be the case as $R_{\alpha}^k(x)\in [-\frac{1}{\underline{v}_{n,j}}, \frac{1}{\underline{v}_{n,j}}]$ and under the assumption $c>1$ we have that $\frac{2}{\underline{v}_{n,j}}\leq \frac{w_n}{q_{n+1}}$ implying that $R_{\alpha}^k(x)$ has to be the point closest to $0$ in the orbit $x,R_{\alpha}(x),\ldots, R_{\alpha}^k(x)$. Hence, $\phi\circ R_{\alpha}^k(x)$ indeed has to be the maximal term. 
\newline

In the following steps we will estimate $\lambda(\underline{B}_n)$. Let us first assume $a_{n+1}>1$ in which case we have $\lfloor \frac{q_{n+1}}{q_n}\rfloor>1$. 
In this case we have 
\begin{align}
 \lambda(\underline{B}_n)
 &=\lambda\left(\bigcup_{j=1}^{\lfloor q_{n+1}/q_n\rfloor}\underline{A}_{n,j}\right)
 =\sum_{j=1}^{\lfloor q_{n+1}/q_n\rfloor}\lambda\left(\underline{A}_{n,j}\right)
 \geq \sum_{j=1}^{\lfloor q_{n+1}/q_n\rfloor-1}\lambda\left(\underline{A}_{n,j}\right)
 =\sum_{j=1}^{\lfloor q_{n+1}/q_n\rfloor-1} \frac{ 2q_n}{\underline{v}_{n,j}}.\label{eq: lambda lower B}
\end{align}
Here, the last two equalities follow by the fact that the unions are disjoint. This can be seen as follows: 
For fixed $j$, we have
$$\min_{0\leq i_1<i_2\leq q_n-1} d(R_{\alpha}^{-i_1-jq_n}(0)\, ,\, R_{\alpha}^{-i_2-jq_n}(0))\leq \frac{1}{q_n}.$$ 
As $\frac{2}{\underline{v}_{n,j}}< \frac{1}{q_n}$, we have that the union defining $\underline{A}_{n,j}$ is indeed disjoint and we obtain the last equality. 
Moreover, we have 
$$\min_{q_n\leq i_1+j_1 q_n<i_2+j_2 q_n\leq q_{n+1}}d(R_{\alpha}^{-i_1-j_1q_n}(0)\, ,\, R_{\alpha}^{-i_2-j_2q_n}(0))\leq \frac{1}{q_{n+1}}.$$ 
If we assume $c>1$, then $\bigcup_{j=1}^{\lfloor q_{n+1}/q_n\rfloor}\underline{A}_{n,j}$ is a disjoint union and we obtain the second equality in \eqref{eq: lambda lower B}.

First assume $n\in\mathcal{N}_2$. Then we have 
\begin{align}
 \frac{2}{\underline{v}_{n,j}}
 &= \frac{2}{3c(j+1) \log q_n +  4c\frac{q_{n+1}}{q_n} (2+\log j)}\notag \\
 &\geq \frac{1}{8c\left( (j+1) \log q_n+  \log q_n (2+\log j)\right)}
 \geq \frac{1}{16c (j+1) \log q_n}.\label{eq: 2over vnj}
\end{align}
By \eqref{eq: lambda lower B} we get for $n$ sufficiently large
\begin{align}
 \lambda(\underline{B}_n)\geq \sum_{j=1}^{\lfloor q_{n+1}/q_n\rfloor-1}\frac{1}{16c (j+1) \log q_n}
 \gg \frac{ \log q_{n+1}-\log q_n}{\log q_n}.\label{eq: lower B N2} 
\end{align}
Next, we assume $a_{n+1}=1$. In this case we have $\lfloor \frac{q_{n+1}}{q_n}\rfloor=1$ and thus
\begin{align}
 \lambda(\underline{B}_n)
 &=\lambda\left(\underline{A}_{n,1}\right)
 = \frac{ 2(q_{n+1}-q_n)}{\underline{v}_{n,j}}
 \gg \frac{ \log q_{n+1}-\log q_n}{\log q_n}.
 \label{eq: lower B a1}
\end{align}
This follows on the one hand by the fact that the union defining $\underline{A}_{n,1}$ is disjoint and on the other hand by the estimate in \eqref{eq: 2over vnj}.

Next, assume $n\in \mathcal{N}_2^c$.
Then we have 
\begin{align*}
 \frac{2}{\underline{v}_{n,j}}
 &\geq \frac{1}{\max\{3c(j+1) \log q_n\, ,\,  4c\frac{q_{n+1}}{q_n} (2+\log j)\}}.
\end{align*}
Note that $3c(j+1) \log q_n\geq  4c\frac{q_{n+1}}{q_n} (2+\log j)$ iff 
\begin{align*}
 \frac{j+1}{2+\log j}\geq  \frac{4 q_{n+1}}{3 q_n\log q_n}. 
\end{align*}
Since $n\in \mathcal{N}_2^c$, we know that the RHS is larger or equal to $\frac{8}{3}$. Since for $j\geq \frac{8}{3}$ we have 
\begin{align*}
 \frac{j+1}{20}\leq \frac{j+1}{2+\log j} \log \left(\frac{j+1}{2+\log j}\right)\leq  20(j+1), 
\end{align*}
we have 
\begin{align*}
 \frac{2}{\underline{v}_{n,j}}
 &\geq
 \begin{cases}
   \frac{q_n}{ 4cq_{n+1} (2+\log j)}    &\text{if } j\leq \frac{ q_{n+1}}{15 q_n\log q_n}\log\left( \frac{4 q_{n+1}}{3 q_n\log q_n}\right)\\
   \\
   \frac{1}{3c(j+1) \log q_n}    &\text{if } j\geq \frac{80 q_{n+1}}{3 q_n\log q_n}\log\left( \frac{4 q_{n+1}}{3 q_n\log q_n}\right).
 \end{cases}
\end{align*}
Setting 
\begin{align*}
 \underline{K}_n&=\left\lfloor \frac{ q_{n+1}}{15 q_n\log q_n}\log\left( \frac{4 q_{n+1}}{3 q_n\log q_n}\right)\right\rfloor 
\qquad\text{ and }\qquad
\underline{K}_n':=\left\lceil \frac{ 80 q_{n+1}}{3 q_n\log q_n}\log\left( \frac{4 q_{n+1}}{3 q_n\log q_n}\right)\right\rceil
\end{align*}
we have 
\begin{align}
 \lambda(\underline{B}_n)
 &\geq  \frac{q_n}{4cq_{n+1}} \sum_{j=1}^{\min \{\underline{K}_n, \lfloor q_{n+1}/q_n\rfloor\}}\frac{1}{2+\log j}
 + \frac{1}{ \log q_n}\sum_{j=\underline{K}_n'}^{\lfloor q_{n+1}/q_n\rfloor}\frac{1}{3c(j+1)}. \label{eq: lambda lower B 2nd case}
\end{align}

For the moment let's assume that $\underline{K}_n\geq 100$, otherwise we simply estimate the first sum by $0$.
Similarly as above, for $n$ sufficiently large, the first of the above sums can be estimated as  
\begin{align*}
 \MoveEqLeft\frac{q_n}{4cq_{n+1}} \sum_{j=1}^{\min \{\underline{K}_n, \lfloor q_{n+1}/q_n\rfloor\}}\frac{1}{2+\log j}\notag\\
 &\geq \frac{q_n }{4cq_{n+1}}\int_{ \mathrm{e}}^{\min \{\underline{K}_n, \lfloor q_{n+1}/q_n\rfloor\}}\left(
 \frac{1}{2+ \log x}-\frac{1}{2+\log^2x}\right)\,\mathrm{d}x\notag\\
 &\geq \frac{q_n }{4cq_{n+1}}\left( \frac{\min\{\underline{K}_n, \lfloor \frac{q_{n+1}}{q_n}\rfloor\}}{2+ \log \left(\min\{\underline{K}_n, \lfloor \frac{q_{n+1}}{q_n}\rfloor\}\right)}- 
 \frac{\mathrm{e}}{3}\right).
\end{align*} 
Let us first assume $\min\{\underline{K}_n, \lfloor \frac{q_{n+1}}{q_n}\rfloor\}= \underline{K}_n$. In this case, setting $h_n=\frac{q_{n+1}}{q_n\log q_n}$ and using $n\in\mathcal{N}_2^c$ we have 
\begin{align}
 \frac{q_n}{4cq_{n+1}} \sum_{j=1}^{\min \{\underline{K}_n, \lfloor q_{n+1}/q_n\rfloor\}}\frac{1}{2+\log j}
 &\geq \frac{q_n }{4cq_{n+1}}\left( \frac{\underline{K}_n}{2+\log \underline{K}_n}- 
 \frac{\mathrm{e}}{3}\right)\notag\\
  &\geq \frac{q_n }{8cq_{n+1}} \frac{\underline{K}_n}{2+\log \underline{K}_n}\notag\\
 &\geq \frac{q_n}{8cq_{n+1}}\frac{  \frac{1}{15} h_n \log\left(\frac{4}{3} h_n\right)-1}{\log\left( \frac{1}{15}h_n\right)+\log\log \left(\frac{4}{3} h_n\right)+2}\notag\\
&\geq \frac{1}{100c\log q_n}\frac{ \log\left(\frac{4}{3} h_n\right)}{\log\left( \frac{1}{15}h_n\right)+\log\log \left(\frac{4}{3} h_n\right)}\notag\\
 &\gg \frac{1}{\log q_n}.\label{eq: lambda lower B1}
\end{align}
On the other hand, for $\epsilon>0$ let 
$ \mathcal{M}_{\epsilon}:=\{n\in\mathbb{N}: q_{n+1}> q_n^{1+\epsilon}\}$. 
Note that if $n\notin \mathcal{M}_{\epsilon}$, then 
\begin{align*}
 \underline{K}_n
 &\leq  \frac{2 q_{n+1}}{15 q_n\log q_n}\log\left( \frac{4 q_{n+1}}{3 q_n\log q_n}\right)
 \leq  \frac{2 q_{n+1}}{15 q_n\log q_n}\log\left( \frac{4 q_{n}^{\epsilon}}{3 \log q_n}\right)\\
 &= \frac{q_{n+1}}{q_n}\frac{2\left(\epsilon \log q_n+ \log( 4/3)-\log\log q_n\right)}{15 \log q_n}
 \leq  \left\lfloor \frac{q_{n+1}}{q_n}\right\rfloor, 
\end{align*}
if $\epsilon$ is sufficiently small and $n$ sufficiently large. 
Let $(m_n)$ be the ordered elements of $\mathcal{M}_{\epsilon}$ which may be infinitely many. In this case, as we shall see, we have
\begin{align}
 q_{m_n}> n^{n\log n},\label{eq: mn estim}
\end{align} 
for $n$ sufficiently large. Let's denote the set of $n$ in $\mathcal{M}_{\epsilon}$ such that \eqref{eq: mn estim} holds by $\mathcal{M}_{\epsilon}'$. \eqref{eq: mn estim} can be seen as follows: 
Assume $q_{m_n}=n^{n\log n}$, then 
\begin{align}
 q_{m_n+1} &> n^{n\log n(1+\epsilon)}
 = (n+1)^{(n+1)\log (n+1)}\cdot \frac{ n^{n\log n(1+\epsilon)}}{(n+1)^{(n+1)\log (n+1)}}\notag\\
 &=(n+1)^{(n+1)\log (n+1)}\left(\frac{n}{n+1}\right)^{n\log n(1+\epsilon)}
 (n+1)^{n\log n(1+\epsilon)- (n+1)\log (n+1)}\label{eq: qn1 estimate}
\end{align}

The second factor can be estimated as
\begin{align}
 \left(\frac{n}{n+1}\right)^{n\log n(1+\epsilon)}
 &>\left(\left(\frac{n}{n+1}\right)^{n+1}\right)^{\log n(1+\epsilon)}
 > \left(\mathrm{e}^{-1}-\epsilon\right)^{\log n(1+\epsilon)},\label{eq: qn1 estimate1} 
\end{align}
for $n$ sufficiently large. 
The third factor in \eqref{eq: qn1 estimate} can be estimated as 
\begin{align*}
 (n+1)^{n\log n(1+\epsilon)- (n+1)\log (n+1)}
 &=(n+1)^{n(\log n-\log (n+1))+\epsilon n\log n-\log (n+1)}
 \geq (n+1)^{\epsilon/2 n\log n}, 
\end{align*}
for $n$ sufficiently large. 
Hence, combining this with \eqref{eq: qn1 estimate} and \eqref{eq: qn1 estimate1} we obtain 
\begin{align*}
 q_{m_n+1}> (n+1)^{(n+1)\log (n+1)}, 
\end{align*}
for $n$ sufficiently large. Noting that $m_{n+1}\geq m_n+1$ and thus $q_{m_{n+1}}\geq q_{m_n+1}$ implies \eqref{eq: mn estim}. 
Moreover, we have 
\begin{align*}
 \sum_{n=j}^{\infty}\frac{1}{\log q_{m_n+1}}
 \leq \sum_{n=j}^{\infty}\frac{1}{\log n^{n\log n}}
 =\sum_{n=j}^{\infty}\frac{1}{ n\log^2 n}<\infty,
\end{align*}
for $j$ sufficiently large
implying
$\sum_{n\in \mathcal{M}_\epsilon'} \frac{1}{\log q_n}<\infty$.
Hence, since the sum over $\mathcal{M}_{\epsilon}$ does at most contribute a finite value even if we use the possibly too large value $\frac{1}{\log q_n}$, we don't need a more precise estimate for $n\in\mathcal{M}_{\epsilon}$. 

Next, we estimate the second sum in \eqref{eq: lambda lower B 2nd case}. 
Note that for $n\in\mathcal{M}_{\epsilon}^c$, we have $2\underline{K}_n'< \lfloor \frac{q_{n+1}}{q_n}\rfloor$, for $n$ sufficiently large.  
Hence, we have 
\begin{align}
 \frac{1}{ \log q_n}\sum_{j=\underline{K}_n'}^{\lfloor q_{n+1}/q_n\rfloor}\frac{1}{3c(j+1)}
 &\gg \frac{1}{\log q_n}\left( \log \left(\frac{ q_{n+1}}{q_n}\right)-
 \log \left( \frac{q_{n+1}}{q_n\log q_n}\log\left( \frac{q_{n+1}}{q_n\log q_n}\right)\right)\right)\notag\\
 &= \frac{\log\log q_n -\log\log \left(\frac{q_{n+1}}{q_n\log q_n}\right) }{\log q_n}.\label{eq: lambda lower B2}
\end{align}
For $n\notin \mathcal{N}_{2}\cap\mathcal{M}_{\epsilon}^c$, \eqref{eq: lambda lower B 2nd case} together with \eqref{eq: lambda lower B1} and \eqref{eq: lambda lower B2} we obtain 
\begin{align}
 \lambda(\underline{B}_n)\gg 
 \frac{\max\left\{ 1\, ,\, \log\log q_n -\log\log \left(\frac{q_{n+1}}{q_n\log q_n}\right) \right\}}{\log q_n}.\label{eq: lower B notN2}
\end{align}
Note that in the estimate of \eqref{eq: lambda lower B1} we were setting the sum equal to $0$ if $\underline{K}_n<100$. However in this case the maximum in \eqref{eq: lower B notN2} is attained at $\log\log q_n -\log\log \left(\frac{q_{n+1}}{q_n\log q_n}\right)$ and the estimate remains true.

The final steps will be devoted to prove that 
\begin{align}
 \lambda(\underline{B}_n\text{ i.o.})=1\label{eq: An io}
\end{align}
which will be done by applying the following Borel-Cantelli lemma by Chandra given in \cite[Rem.~1]{chandra}.

\begin{lemma}\label{lem:chandra}
 Let $(D_n)$ be a sequence of events in a probability measure space $(X,\mathbb{P})$ such that
 \begin{align*}
  \sum_{n=1}^{\infty}\P(D_n)=\infty.
 \end{align*}
 Additionally, assume that
 for all $i<j$ we have
 \begin{align}
 \P(D_i \cap D_j)- \P(D_i)\P(D_j)\leq  \rho(|i-j|)[\P(D_i)+\P(D_{i+1})+ \P(D_j)+ \P(D_{j+1})],\label{eq: cond on Bi}
 \end{align}
 with $\rho$ fulfilling 
 \begin{align*}
  \liminf_{n\to\infty}\left(\frac{\sum_{m=1}^{n-1} \rho(m)}{\sum_{m=1}^n\P(D_m)}\right)=0. 
 \end{align*}
 Then $\P(\sum_{n=1}^{\infty}\chi_{D_n}=\infty)=1$. 
\end{lemma}

In order to be able to apply this lemma, we first need some statements about the independence of the sets $\underline{B}_n$ which will be given in the following:

\begin{lemma}\label{lem:add} 
Let 
\begin{align*}
 J_n&:= \bigcup_{j=1}^{\lfloor q_{n+1}/q_n\rfloor-1} \left[ R_{\alpha}^{-jq_n}(0)-\frac{1}{\underline{v}_{n,j}}\, ,\, R_{\alpha}^{-jq_n}(0) + \frac{1}{\underline{v}_{n,j}}\right]\quad\text{ and }\\
 J_n'&:=\left[ R_{\alpha}^{-\lfloor q_{n+1}/q_n\rfloor q_n}(0)-\frac{1}{\underline{v}_{n,\lfloor q_{n+1}/q_n\rfloor}}\, ,\, R_{\alpha}^{-\lfloor q_{n+1}/q_n\rfloor q_n}(0) + \frac{1}{\underline{v}_{n,\lfloor q_{n+1}/q_n\rfloor}}\right].
\end{align*}
Then for $c$ and $n$ sufficiently large
\begin{enumerate}[label=\textbf{(B.\arabic*)}]
\item\label{B1} the set $J_{n}$ is a union of disjoint intervals $\{I_{n,k}\}$, $0 \leq k\leq \lfloor \frac{q_{n+1}}{q_n}\rfloor$, where all the $I_{n,k}$
satisfy $|I_{n,k}|\in [\min\{ \frac{1}{6cq_{n+1}\log q_n}, \frac{1}{8c q_{n+1}(\log q_{n+1}-\log q_n)}\}, \frac{6}{q_n}]$ and $\diam(J_{n})\leq \frac{6}{q_n}$;
\item\label{B2} for each $i\in \{0,\ldots, q_{n-1}-1\}$ we have $|J'_{n}|\in  [\min\{ \frac{1}{6cq_{n+1}\log q_n}, \frac{1}{8c q_{n+1}(\log q_{n+1}-\log q_n)}\}, \frac{6}{q_n}]$; 
\item\label{B3} 
we have  $$\underline{B}_n=\bigcup_{i=0}^{q_n-1} R^{-i}_\alpha(J_n)\cup \bigcup_{i=0}^{q_{n-1}-1} R^{-i}_\alpha(J'_n);$$

\item\label{B4} the sets $\{R^{-i}_\alpha(J_n)\}$ and $\{R^{-i}_\alpha(J'_n)\}$ are pairwise disjoint.
\end{enumerate}
\end{lemma}
\begin{proof} Note that property \ref{B2} is immediate from the definition of $\underline{v}_{n,\lfloor q_{n+1}/q_n\rfloor}$ (see 
\eqref{vnj}). 
Moreover property \ref{B3} is straightforward from the definition of the set $\underline{B}_n$ - we have just exchanged the order of the involved unions and used $q_{n+1}-1-\lfloor \frac{q_{n+1}}{q_n}\rfloor q_n= q_{n-1}-1$.

 We will now show \ref{B1}. Note that just from the definition of $J_n$ it follows that $J_n$ is a union of intervals - we will only need to show that they are disjoint.
 
  Denote 
 $$I_{n,j}:=\left[ R_{\alpha}^{-jq_n}(0)-\frac{1}{\underline{v}_{n,j}}\, ,\, R_{\alpha}^{-jq_n}(0) + \frac{1}{\underline{v}_{n,j}}\right].$$ 
 
We have that 
\begin{equation}
 \max_{j, j'\in  \{1,\ldots, \lfloor q_{n+1}/q_n\rfloor-1\}, j\neq j'}d(R_{\alpha}^{-jq_n}(0),R_{\alpha}^{-j'q_n}(0))= \frac{w_n}{q_{n+1}}\geq \frac{1}{2q_{n+1}}.\label{eq: max j j'}
\end{equation}
 On the other hand, for $c\geq 2$, we have 
 $|I_{n,j}|=\frac{2}{\underline{v}_{n,j}}\leq \frac{1}{2q_{n+1}}$, for all $j \in \{ 1,\ldots, \lfloor \frac{q_{n+1}}{q_n}\rfloor-1\}$ implying disjointness.

Note that trivially 
\begin{align*}
  |I_{n,j}|\geq 2 \min_{j< \lfloor q_{n+1}/q_n\rfloor} \frac{1}{\underline{v}_{n,j}}
  \geq \min\Big\{ \frac{1}{6cq_{n+1}\log q_n}, \frac{1}{8c q_{n+1}(\log q_{n+1}-\log q_n)}\Big\},
 \end{align*}
for $n$ sufficiently large.
 Moreover,  since $\max_{j<\lfloor q_{n+1}/q_n\rfloor}d(R_{\alpha}^{-jq_n}(0),0)\leq j|q_n\alpha|\leq \frac{2}{q_n}$ it follows that $I_{n,j}\subset J_n\subset [-\frac{3}{q_n}, \frac{3}{q_n}]$, which completes the proof of \ref{B1}.

Finally, in order to prove \ref{B4}, remember \eqref{eq: max j j'}.
As any of the intervals defining $J_n$ and $J_n'$ have a diameter of at most $\frac{1}{2q_{n+1}}$ and they are all symmetric around $R_{\alpha}^{-j}(0)$, we obtain \ref{B4}. 
\end{proof}
The following lemma is crucial in obtaining independence properties of the collection $\{\underline{B}_n\}$.
\begin{lemma}\label{lem:add2} We have 
$$
|\lambda(\underline{B}_i\cap \underline{B}_{i+m})- \lambda(\underline{B}_i)\lambda(\underline{B}_{i+m})|={\rm O}\Big(\frac{q_{i+1}}{q_{i+m-2}}\lambda(\underline{B}_{i+m})\Big),
$$
where $\rm O$ does not depend on $i$ or $m$.
\end{lemma}
\begin{proof}
 We will use Lemma \ref{lem:add} together with the Denjoy-Koksma inequality. Consider $\underline{B}_{i+m}\cap R_\alpha^{-\ell}(J_i)$ for 
$\ell< \lfloor \frac{q_{n+1}}{q_n}\rfloor$. To estimate the measure of this set we  will estimate measures of the sets 
$(\bigcup_{s=0}^{q_{i+m}-1}R^{-s}_\alpha(J_{i+m}))\cap R_{\alpha}^{-\ell} (I_{i,k})$ and  $(\bigcup_{s=0}^{q_{i+m-1}-1}R^{-s}_\alpha(J'_{i+m}))\cap R_{\alpha}^{-\ell} (I_{i,k})$. To do this, note that by \ref{B1} the set $R_{\alpha}^{-\ell} (I_{i,k})$ is an interval of length $\frac{1}{\underline{v}_{n,k}}\leq \frac{1}{4cq_{n+1}}$;
we denote it by $[a,b]$. 
Moreover, $\diam(J_{i+m})\leq \frac{6}{q_{i+m}}$. Take any $x_{i+m}\in J_{i+m}$. Then 
\begin{align*}
 \MoveEqLeft\lambda\Big(\bigcup_{ s=q_{i+m-1}}^{q_{i+m}-1}R^{-s}_\alpha(J_{i+m})\cap [a,b]\Big)\\
 &\leq \lambda(J_{i+m})\, \#\big\{s { \in [q_{i+m-1},q_{i+m})\cap\N}\;:\; R^{-s}_\alpha(x_{i+m})\in [a- \diam(J_{i+m}),b+\diam(J_{i+m})]\big\} 
\end{align*}
and 
\begin{align*}
 \MoveEqLeft\lambda\Big(\bigcup_{s=q_{i+m-1}}^{q_{i+m}-1}R^{-s}_\alpha(J_{i+m})\cap [a,b]\Big)\\
 &\geq \lambda(J_{i+m})\, \#\big\{s{ \in [q_{i+m-1},q_{i+m})\cap\N}\;:\; R^{-s}_\alpha(x_{i+m})\in [a+ \diam(J_{i+m}),b-\diam(J_{i+m})]\big\}.
\end{align*}
We will estimate both cardinalities using the Denjoy-Koksma inequality. Let 
$$f^+=\chi_{[a- \diam(J_{i+m})\,,\,b+\diam(J_{i+m})]},$$ 
then 
\begin{align*}
 \MoveEqLeft\#\big\{s  \in [q_{i+m-1},q_{i+m})\cap\N\;:\; R^{-s}_\alpha(x_{i+m})\in [a- \diam(J_{i+m}),b+\diam(J_{i+m})]\big\}\\
 &= \sum_{s=0}^{q_{i+m}-1}f^+(R^{-s}_\alpha(x_{i+m})).
\end{align*}
By the Denjoy-Koksma inequality we have
\begin{align*}
 \MoveEqLeft\bigg|\sum_{s=q_{i+m-1}}^{q_{i+m}-1}f^+(R^{-s}_\alpha(x_{i+m}))- \left(q_{i+m}-q_{i+m-1}\right)\lambda(f^+)\bigg|\\
 &\leq \bigg|\sum_{s=0}^{q_{i+m}-1}f^+(R^{-s}_\alpha(x_{i+m}))- q_{i+m}\lambda(f^+)\bigg|
 +\bigg|\sum_{s=0}^{q_{i+m-1}-1}f^+(R^{-s}_\alpha(x_{i+m}))- q_{i+m-1}\lambda(f^+)\bigg|\\
 &<2\Var(f^+)\leq 8.
\end{align*}

Moreover, by the bound on $\diam(J_{i+m})$, we have $\lambda(f^+)\leq \lambda([a,b])+ 12 q_{i+m}^{-1}$. Using an analogous estimate for the lower bound (and putting both estimates together) yields
\begin{align*}
 \lambda\Big(\bigcup_{ s=q_{i+m-1}}^{q_{i+m}-1}R^{-s}_\alpha(J_{i+m})\cap [a,b]\Big)
 &={ \left(q_{i+m}-q_{i+m-1}\right)}\lambda([a,b])\lambda(J_{i+m})+{\rm O}(\lambda(J_{i+m}))\\
 &= \lambda\Big({ \bigcup_{s=q_{i+m-1}}^{q_{i+m}-1}}R^{-s}_\alpha(J_{i+m})\Big)\lambda([a,b])+{\rm O}(\lambda(J_{i+m})),
\end{align*}
where the second equality follows from \ref{B4} in Lemma \ref{lem:add}.

Summing this over all $[a,b]=R_{\alpha}^{-\ell} (I_{i,k})$ and also $[a,b]=R_\alpha^{-\ell}(J_{i}')$, we get 
$$
\lambda\Big({\bigcup_{s=q_{i+m-1}}^{q_{i+m}-1}}R^{-s}_\alpha(J_{i+m})\cap \underline{B}_i\Big)
=\lambda\Big({ \bigcup_{s=q_{i+m-1}}^{q_{i+m}-1}}R^{-s}_\alpha(J_{i+m})\Big)\lambda( \underline{B}_i)+{\rm O}({\left(q_{i+1}-q_i\right)}\lambda(J_{i+m})).
$$
By an analogous reasoning,  
$$
\lambda\Big({ \bigcup_{s=q_{i+m-1}}^{q_{i+m-1}-1}}R^{-s}_\alpha(J'_{i+m})\cap \underline{B}_i\Big)=\lambda\Big({ \bigcup_{s=q_{i+m-1}}^{q_{i+m-1}-1}}R^{-s}_\alpha(J'_{i+m})\Big)\lambda( \underline{B}_i)+{\rm O}({\left(q_{i+1}-q_i\right)}\lambda(J'_{i+m})).
$$
Putting the two above together and using \begin{align}
 \left(q_{i+1}-q_i\right)\lambda(J_{i+m})
 &= \frac{q_{i+1}-q_i}{q_{i+m}-q_{i+m-1}}\,\lambda\bigg(\bigcup_{s=q_{i+m-1}}^{q_{i+m}-1}R^{-s}_\alpha(J_{i+m})\bigg)\notag\\
 &\leq \frac{q_{i+1}}{q_{i+m-2}}\,\lambda\bigg(\bigcup_{s=q_{i+m-1}}^{q_{i+m}-1}R^{-s}_\alpha(J_{i+m})\bigg)\label{eq: qi lambdaJi}
\end{align}
and ${ (q_{i+1}-q_i)}\lambda(J'_{i+m})= \frac{ q_{i+1}-q_i}{q_{i+m-1}}\lambda(\bigcup_{s=0}^{q_{i+m-1}-1}R^{-s}_\alpha(J'_{i+m}))$,
we get 
$$
\lambda(\underline{B}_i\cap \underline{B}_{i+m})= \lambda(\underline{B}_i)\lambda(\underline{B}_{i+m})+{\rm O}\Big(\frac{q_{i+1}}{ q_{i+m-2}}\lambda(\underline{B}_{i+m})\Big).
$$
This finishes the proof.
\end{proof}

Finally, we are in the position to prove \eqref{eq: An io}.
We notice that \eqref{eq: sum 1logqn} together with \eqref{eq: lower B N2}, \eqref{eq: lower B a1}, and \eqref{eq: lower B notN2} implies 
$\sum_{n=1}^{\infty}\lambda(\underline{B}_n)=\infty$. 
 Moreover, this also implies that there is at least one $u\in {\ \{0,1,2,3\}}$ such that 
 $\sum_{n=1}^{\infty}\lambda(\underline{B}_{ 4n+u})=\infty$.
Then, by Lemma \ref{lem:add2}, there exists $K>0$ such that for all $j>i$ we have 
\begin{align*}
 |\lambda(\underline{B}_{ 4i+u}\cap \underline{B}_{4j+u})- \lambda(\underline{B}_{ 4i+u})\lambda(\underline{B}_{ 4j+u})|
 &\leq K\frac{q_{ 4i+u+1}}{q_{ 4j+u-2}}\lambda(\underline{B}_{ 4j+u})\\
 &\leq K\sup_{\ell\in \N}\frac{q_{\ell}}{q_{ \ell+4(j-i)-3}}\lambda(\underline{B}_{ 4j+u}).
\end{align*}
Setting $\rho(n)= K\sup_{\ell\in \N}\frac{q_{\ell}}{q_{ \ell+4n-3}}$ and noting that the $q_n$ grow at least exponentially fast, we have $\sum_{n=1}^{\infty}\rho(n)<\infty$.

Hence, \eqref{eq: cond on Bi} holds with $\rho$ for $D_m=\underline{B}_{ 4m+u}$ and we obtain \eqref{eq: An io}. 
\end{proof}

\subsection{Proof of Theorem \ref{prop: measzeroset}}
\begin{proof}[Proof of Theorem \ref{prop: measzeroset}] 
\eqref{prop: HD dimension} immediately implies \eqref{eq: Wn finite small}. Hence, we only focus on proving \eqref{prop: HD dimension}.

In order to prove the lower bound we use the following statement by {\L}uczak, \cite{luczak}.
\begin{lemma}\label{lem: luczak}
Let 
\begin{align*}
 \widetilde{\Xi}(b,c)=\{\alpha\in [0,1]:a_n =c^{b^n}\text{ for all }n\in\mathbb{N}\}.
\end{align*}
Then for all $b,c\in (1,\infty)$ we have 
\begin{align*}
 \dim_{\mathrm{H}}(\widetilde{\Xi}(b,c))
 =\frac{1}{b+1}.
\end{align*}
\end{lemma}

Let $b=1+\epsilon$ and $c>1$. Moreover, we note that $W_n\ll \frac{\log\log q_n}{\log q_n}$ and the function $\frac{\log\log(\cdot)}{\log(\cdot)}$ is eventually monotonically decreasing. 
Under the assumption that $\alpha\in \widetilde{\Xi}(b,c)$ and noting then that $q_n> a_n= c^{(1+\epsilon)^n}$ we have
\begin{align*}
 W_n(\alpha)\ll \frac{\log\log c+ n\log (1+\epsilon) }{(1+\epsilon)^n\log c}
\end{align*}
and thus $\sum_{n=1}^{\infty}W_n(\alpha)<\infty$. 
Hence, for all $\epsilon>0$ and all $c>1$ we have that 
$ \widetilde{\Xi}(1+\epsilon, c)\subset \{\alpha: \sum_{n=1}^{\infty} W_n(\alpha)<\infty\}$ and
Lemma \ref{lem: luczak} implies that 
\begin{align*}
 \dim_{\mathrm{H}}\Big(\alpha: \sum_{n=1}^{\infty} W_n(\alpha)<\infty\Big)\geq \dim_{\mathrm{H}}\big(\widetilde{\Xi}(1+\epsilon, c)\big)= \frac{1}{2+\epsilon}.
\end{align*}
Since $\epsilon$ was arbitrary, we get the lower bound for the Hausdorff dimension.

In order to prove the upper bound, we note that Fan, Liao, Wang and Wu, see \cite[Theorem 1.2.\ (1)]{fanliauwangwu}, proved that
\begin{align}
 \dim_{\mathrm{H}}\left(\alpha: \lim_{n\to\infty}\frac{\log q_n}{n}=\infty\right)=\frac{1}{2}. \label{eq: dimH upper}
\end{align}

For the following, let $\mathcal{D}=\{n: a_{n+1}= 1\text{ and }a_n\neq 1\}$. Note that for $n\in \mathcal{D}^c$ we have $W_n\gg \frac{1}{\log q_n}$. This can be seen as follows: If $q_{n+1}\geq 2q_n\log q_n$, this is immediately the case; so assume $q_{n+1}< 2q_n\log q_n$.
If $a_{n+1}\geq 2$, then $q_{n+1}\geq 2 q_{n}$ and thus $\log q_{n+1}-\log q_n\geq \log 2$ implying that $W_n\geq \frac{1}{\log q_n}$. 
On the other hand, if $a_{n+1}=a_n=1$, then $q_n=q_{n-1}+q_{n-2}$ and
$q_{n+1}=q_n+q_{n-1}=2q_{n-1}+q_{n-2}$ which implies $\log q_{n+1}-\log q_n\leq \log (\frac{3}{2})$.

Moreover, we have that $\sum_{n=1}^{\infty} W_n=\infty$ already implies $\sum_{n\in\mathcal{D}^c}W_n=\infty$. This can be seen as follows: If $n\in\mathcal{D}$, then 
\begin{align}
 \frac{\log q_{n+1}-\log q_n}{\log q_n}
 &= \frac{\log \left(1+\frac{q_{n-1}}{q_n}\right)}{\log q_n}
 \leq \frac{\log \left(1+\frac{q_{n-1}}{a_nq_{n-1}+q_{n-2}}\right)}{\log q_n}\leq \frac{1}{\log q_n}.\label{eq: Destim1}
\end{align}
Moreover, $q_n=q_{n+1}-q_{n-1}$ and $q_{n+1}=(a_{n-1}+1)q_{n-1}+q_{n-2}$. This implies 
$$q_n=q_{n+1}\left(1-\frac{q_{n-1}}{(a_{n-1}+1)q_{n-1}+q_{n-2}}\right)\geq \frac{2}{3}q_{n+1}$$
which together with \eqref{eq: Destim1} implies 
\begin{align*}
 \frac{\log q_{n+1}-\log q_n}{\log q_n}
 &\leq \frac{1}{\log q_{n+1}+\log \big( \frac{2}{3}\big)}\leq \frac{2}{\log q_{n+1}}, 
\end{align*}
for $n$ sufficiently large, say for $n\geq M$. Thus, noting that 
$n\in\mathcal{D}$ implies $n+1\in\mathcal{D}^c$, we have $W_n\leq 2W_{n+1}$, for all $n\in \mathcal{D}\cap\mathbb{N}_{\geq M}$. 
Hence, $\sum_{n\in\mathcal{D}^c\cap\mathbb{N}_{\geq M}}W_n\geq \frac{1}{3}\sum_{n=M}^{\infty} W_n=\infty$ as claimed.

We have $W_n\ll \frac{1}{\log q_n}$, for all $n\in \mathcal{D}^c$. As $\sum_{n=1}^{\infty}W_n=\infty$ implies $\sum_{n\in \mathcal{D}^c} W_n=\infty$, it is enough to consider $n\in \mathcal{D}^c$. 

However, 
$\sum_{n=1}^{\infty}\frac{1}{\log q_n}<\infty$ implies $\frac{\log q_n}{n}\to\infty$ and thus
$$\left\{\alpha: \lim_{n\to\infty}\frac{\log q_n}{n}=\infty\right\}\supset \left\{\sum_{n=1}^{\infty} W_n<\infty\right\}.$$ This together with \eqref{eq: dimH upper} gives the upper bound.
\end{proof}

\subsection{Proof of Theorem \ref{thm: phi beta}}
\begin{proof}[Proof of Theorem \ref{thm: phi beta}]
 The proof can be done in a similar, but easier way than the proof of Theorem \ref{thm: main rot}. First note that 
 $\int_{1/q_n}^{ 1-1/q_n} x^{-\gamma}\mathrm{d}x\asymp q_n^{\gamma-1}$. 
Hence, by the Denjoy-Koksma inequality we have 
\begin{align}
 S_{q_n}(\widetilde{\phi}_{\gamma})\ll q_n^{\gamma}+\widetilde{\phi}_{\gamma}(x_{\min, q_n}). \label{eq: Sqn beta}
\end{align}
Set 
\begin{align*}
 E_{n,K}= \bigcup_{\ell=q_{n-1}}^{q_n-1} \Gamma_{n,\ell,K}=\bigcup_{\ell=q_{n-1}}^{q_n-1}\left[ R_\alpha^{-\ell}(0)-\frac{1}{Kq_n} \, ,\, R_\alpha^{-\ell}(0) +\frac{1}{Kq_n} \right].
\end{align*}

Moreover, as the union defining $E_{n,K}$ is disjoint, for $K$ sufficiently large, we have \begin{align*}
 \lambda(E_{n,K})=\frac{2}{K}\frac{q_n-q_{n-1}}{q_{n}}
 =\frac{2}{K}\left(1-\frac{1}{a_n}+\frac{q_{n-2}}{a_nq_{n}}\right).
\end{align*}
If there are infinitely many $n$ such that $a_n>1$, it immediately follows that $\sum_{n=1}^{\infty}\lambda(E_{n,K})=\infty$.
On the other hand, if $a_n=1$ eventually, then for $n$ sufficiently large we have 
\begin{align*}
 \lambda(E_{n,K})
 =\frac{2}{K}\frac{q_{n-2}}{2q_{n-2}+q_{n-3}}>\frac{2}{3K}
\end{align*}
and $\sum_{n=1}^{\infty}\lambda(E_{n,K})=\infty$ follows as well.

In the next steps, we prove that for any given $K>0$ there exists $K'>0$ such that
 \begin{align}
  \widetilde{C}_{n}:=\left\{x: \exists k\in ( q_{n-1},q_n]: \phi\circ R_{\alpha}^k(x)> K S_{k}(x) \right\} \supset E_{n, K'}.\label{eq: Cnj supset EnK}
 \end{align}
First note that the statement $ x\in E_{n,K'}$ is equivalent to
$R_{\alpha}^k(x)\in [-\frac{1}{K'q_n}, \frac{1}{K'q_n}]$ for some $k\in \{q_{n-1},\ldots, q_n-1\}$. This implies that there exists $k\in \{q_{n-1},\ldots, q_n-1\}$ such that 
$\phi\circ R_{\alpha}^k(x)>(K'q_n)^{\gamma}$.
However, for $k\leq q_n$, we can conclude by \eqref{eq: Sqn beta} that there exists $K''>0$ such that $S_{k}(x)-\phi(x_{\min,k})\leq  K'' q_n^{\gamma}$, for $n$ sufficiently large.
Hence, \eqref{eq: Cnj supset EnK} holds. 

By an analogous calculation as in Lemma \ref{lem:add2} considering the sets $\Gamma_{i,\ell,K'}$ instead of $R_{\alpha}^{-\ell}(I_{i,k})$, the sets $\Gamma_{i+m,\ell,K'}$ instead of $J_{i+m}$ and noting that also the $\Gamma_{i,\ell,K'}$ are pairwise disjoint
we obtain 
$$
|\lambda(E_{i,K'}\cap E_{i+m,K'})- \lambda(E_{i,K'})\lambda(E_{i+m,K'})|={\rm O}\Big(\frac{q_{i+1}}{q_{i+m-2}}\lambda(E_{i+m,K'})\Big).
$$
Remember that $\sum_{n=1}^{\infty}\lambda(E_{n,K})=\infty$.
Using the same construction as in the proof Theorem \ref{thm: main rot} 
we obtain by using  Lemma \ref{lem:chandra} that $\lambda(\sum_{n=1}^{\infty}\chi_{E_{n,K'}}=\infty)=1$.
\end{proof}

\subsection{Proof of Proposition \ref{prop: full measure set}}
\begin{proof}[Proof of Proposition \ref{prop: full measure set}] Let $K$ be the set of $\alpha\in \T$ such that $\sum_{n=1}^\infty \frac{1}{\log q_n}=\infty$. It follows from the Khintchine-L\'evy theorem that such a set of $\alpha$ is of full measure. For the following, we always assume that $\alpha\in K$. Let $\epsilon_n:=\frac{2}{\log \log q_n}$. Consider now the following sets
$$C_n:=\bigcup_{j=-q_n}^{q_n-1} R_\alpha^{-j}\Big[-\frac{1}{\epsilon_n q_n\log q_n},\frac{1}{\epsilon_n q_n\log q_n}\Big]$$ 
and 
$$D_n:=\bigcup_{j=0}^{q_n-1} R_\alpha^{-j}\Big[-\frac{\epsilon_n}{q_n\log q_n},\frac{\epsilon_n}{q_n\log q_n}\Big].$$ 

Consider  
$$
Q_1:=\{(x,\beta)\in \T^2\;:\; (x,\beta)\in D_n\times C_n^c\text{ for infinitely many }n\in \N\}.
$$
Note that if $(x,\beta)\in D_n\times C_n^c$, then $\{x+j\alpha\}_{j=0}^{q_n-1}\cap \Big[-\frac{\epsilon_n}{q_n\log q_n},\frac{\epsilon_n}{q_n\log q_n}\Big]\neq \emptyset$ (in fact there is just one $j_0 \in [0, q_n-1]$ which is in the intersection) and  $\{x-\beta+j\alpha\}_{j=0}^{q_n-1}\cap \Big[\frac{-1}{2\epsilon_nq_n\log q_n},\frac{1}{2\epsilon_nq_n\log q_n}\Big]=\emptyset$. Indeed, the first statement follows directly from the fact that $x\in D_n$ (and for uniqueness we use Diophantine properties of orbits of length $q_n$). For the second note that if there was a $j_1\in [0,q_n-1]$ such that $x-\beta+j_1\alpha$ belonged to the set $\Big[\frac{-1}{2\epsilon_nq_n\log q_n},\frac{1}{2\epsilon_nq_n\log q_n}\Big]$, then  
 $\beta+(j_0-j_1)\alpha=(x+j_0\alpha)-(x-\beta+j_1\alpha)\in \Big[\frac{-1}{\epsilon_nq_n\log q_n},\frac{1}{\epsilon_nq_n\log q_n}\Big]$ which contradicts the fact that $\beta\in C_n^c$. Therefore, for $x\in D_n$,  
$$
S_{j_0+1}(\phi)(x)\geq \phi(x+j_0\alpha)\geq \frac{q_n\log q_n}{\epsilon_n}
$$ 
and by Lemma \ref{lem: DK adapted} for $(x,\beta)\in D_n\times C_n^c$
$$
S_{j_0+1}(\phi)(x-\beta)\leq S_{q_n}(\phi)(x-\beta)< 2q_n\log q_n + q_n+ 2\epsilon_n q_n\log q_n. 
$$
Therefore, $\Theta_{j_0+1}^\beta(x)> \frac{1}{10\epsilon_n}$, for $n$ sufficiently large. In particular, for $(x,\beta)\in  Q_1$, we have $\limsup_{N\to\infty}\Theta_N^\beta(x)=\infty$. Analogously, defining the set 
$$
 Q_2:=\{(x,\beta)\in \T^2\;:\; (x-\beta,\beta)\in D_n\times C_n^c\text{ for infinitely many }n\in \N\},
$$
we get that for $(x,\beta)\in  Q_2$, $\liminf_{N\to\infty}\Theta_N^\beta(x)=0$. Consider the set $Q_1\cap Q_2$. We will show that $\lambda_2( Q_1\cap Q_2)=1$. (Here and in the following we will distinguish between $\lambda_2$ and $\lambda_1$ depending if we consider to two- or one-dimensional Lebesgue measure.) This by Fubini's theorem implies that there is a full measure set $G$ of $\beta$ such that for each $\beta\in G$ there is a full measure set $H_\beta$ of $x$ such that $(x,\beta)\in Q_1\cap Q_2$. In particular taking $Z=K\times G$ finishes the proof of the proposition.

So it remains to show that $\lambda_2( Q_i)=1$ for $i=1,2$.
Since the proof is similar in both cases, we will only show that $\lambda_2( Q_2)=1$ (in fact the proof for $Q_2$ is slightly more involved as we have a twisted condition $x-\beta$ in the first coordinate).
First, for $j\in \{0,\ldots, q_{n}-1\}$ define
\begin{align*}
\widetilde{C}_{n,j}:=R_\alpha^{-j}\Big[-\frac{1}{\epsilon_n q_n\log q_n},\frac{1}{\epsilon_n q_n\log q_n}\Big]\qquad\text{and}\qquad
 \widetilde{C}_n:=\bigcup_{j=0}^{q_n-1} \widetilde{C}_{n,j}\subset C_n
\end{align*}
and let 
$$L_n:=\{(x,\beta)\in \T^2\;:\; (x-\beta,\beta)\in D_n\times \widetilde{C}_n^c\}.$$
Notice first that by Fubini's theorem and translational invariance of Lebesgue measure, $\lambda_2(L_n)=\lambda_1(\widetilde{C}_n^c)\lambda_1(D_n)$. Moreover, for any $i,j$, $\lambda_2(L_i\cap L_j)=  \lambda_1(\widetilde{C}_i^c\cap \widetilde{C}_j^c)\lambda_1(D_i\cap D_j)$. 
This implies
\begin{align}
 |\lambda_2(L_i\cap L_j)- \lambda_2(L_i)\lambda_2(L_j)|
 &=|\lambda_1(\widetilde{C}_i^c\cap \widetilde{C}_j^c)\lambda_1(D_i\cap D_j)- \lambda_1(\widetilde{C}_i^c)\lambda_1(\widetilde{C}_j^c)\lambda_1(D_i)\lambda_1(D_j)|\notag\\
  &\leq |\lambda_1(\widetilde{C}_i^c\cap \widetilde{C}_j^c)\lambda_1(D_i\cap D_j)
  - \lambda_1(\widetilde{C}_i^c)\lambda_1(\widetilde{C}_j^c)\lambda_1(D_i \cap D_j)|\notag\\
  &\qquad+|\lambda_1(\widetilde{C}_i^c)\lambda_1( \widetilde{C}_j^c)\lambda_1(D_i\cap D_j)- \lambda_1(\widetilde{C}_i^c)\lambda_1(\widetilde{C}_j^c)\lambda_1(D_i)\lambda_1(D_j)|\notag\\
 &\leq|\lambda_1(\widetilde{C}_i^c\cap \widetilde{C}_j^c)
 - \lambda_1(\widetilde{C}_i^c)\lambda_1(\widetilde{C}_j^c)|\,\lambda_1( D_j)\notag\\
 &\qquad + \lambda_1( \widetilde{C}_j^c)\,|\lambda_1(D_i\cap D_j)- \lambda_1(D_i)\lambda_1(D_j)|.\label{eq: Ln dep}
\end{align}
Notice further that 
\begin{align}
 |\lambda_1(\widetilde{C}_i^c\cap \widetilde{C}_j^c)
 - \lambda_1(\widetilde{C}_i^c)\lambda_1(\widetilde{C}_j^c)|
 =|\lambda_1(\widetilde{C}_i\cap \widetilde{C}_j)
 - \lambda_1(\widetilde{C}_i)\lambda_1(\widetilde{C}_j)|.\label{eq: Ci Cic}
\end{align}
Notice that, for $n$ sufficiently large, the intervals building $\widetilde{C}_n$ are disjoint (indeed we have $\min_{j,j'\in \{0, \ldots, q_{n}-1\}, j\neq j'} R_\alpha^{-j}(0)\geq \frac{1}{2q_n}$ and 
$\diam(\widetilde{C}_{n,j})=\frac{2}{\epsilon_n q_n \log q_n}$, for each $j$).

Next, we follow essentially the same steps as in Lemma \ref{lem:add2} looking at $\widetilde{C}_{i+m,0}$ instead of $J_{i+m}$ and at an interval $[a,b]=\widetilde{C}_{i,\ell}$ instead of $R_{\alpha}^{-\ell}(I_{i,k})$. Moreover, the union ranges from $s\in \{0,\ldots, q_{i+m}-1\}$ instead of $s\in \{q_{i+m-1},\ldots, q_{i+m}-1\}$, so in the analogous line to \eqref{eq: qi lambdaJi} we have a factor $\frac{q_{i+1}}{q_{i+m+1}}$ instead of $\frac{q_{i+1}-q_i}{q_{i+m-1}}$. Thus, we obtain 
\begin{align}
 \lambda_1(\widetilde{C}_i\cap \widetilde{C}_{i+m})= \lambda(\widetilde{C}_i)\lambda_1(\widetilde{C}_{i+m})+{\rm O}\Big(\frac{q_{i+1}}{q_{ i+m+1}}\lambda_1(\widetilde{C}_{i+m})\Big).\label{eq: Ci dep}
\end{align}
Using an analogous argumentation we obtain 
\begin{align}
 \lambda_1(D_i\cap D_{i+m})= \lambda_1(D_i)\lambda_1(D_{i+m})+{\rm O}\Big(\frac{q_{i+1}}{q_{i+m+1}}\lambda_1(D_{i+m})\Big).\label{eq: Di dep}
\end{align}

We aim to use Lemma \ref{lem:chandra} for $L_n$. Combining \eqref{eq: Ln dep}, \eqref{eq: Ci Cic}, \eqref{eq: Ci dep}, and \eqref{eq: Di dep} we obtain that there exists $\kappa>0$ such that for all $i<j$ we have 
\begin{align*}
 |\lambda_2(L_i\cap L_j)- \lambda_2(L_i)\lambda_2(L_j)|
 &\leq \kappa { \frac{q_{i+1}}{q_{j+1}}} \lambda_1(\widetilde{C}_j)\lambda_1(D_j)
 \leq \kappa {\frac{q_{i+1}}{q_{j+1}}} \lambda_1(\widetilde{C}_j^c)\lambda_1(D_j).
\end{align*}
 Noting that the $q_n$ decay exponentially fast and that there exists $N\in\N$ with $\sum_{n=N}^{\infty}\lambda_2(L_n)\geq \sum_{n=N}^{\infty}\frac{\lambda_1(D_n)}{2}\geq \sum_{n=N}^{\infty}\frac{1}{\log q_n}=\infty$, we obtain that 
$\lambda_2( L_n\text{ for infinitely many }n\in \N)=1$. 
As $\{L_n\text{ for infinitely many }n\in \N\}\subset Q_1$ this gives the statement of the proposition.
\end{proof}

\section*{Acknowledgements}
A.K.\ was partially supported by NSF grant DMS-2247572 and by a
grant from the priority research area SciMat under the Strategic Programme Excellence Initiative at Jagiellonian University.
T.S.\ was supported by a
grant from the priority research area SciMat under the Strategic
Programme Excellence Initiative at Jagiellonian University and by the European Union’s Horizon Europe
research and innovation programme under the Marie Sklodowska-Curie grant agreement ErgodicHyperbolic - 10115118.
Moreover, she thanks the hospitality of the Universities of Maryland, Vienna and Zurich where part of this work was done.

\end{document}